\documentclass[11pt]{scrartcl}

\usepackage[normalem]{ulem}

\usepackage{amsmath,amssymb,latexsym,soul,amsthm}
\usepackage{cite,amsfonts}
\usepackage{color,enumitem,graphicx}
\usepackage[colorlinks=true,urlcolor=blue,
=red,linkcolor=blue,linktocpage,pdfpagelabels,
bookmarksnumbered,bookmarksopen]{hyperref}
\usepackage[english]{babel}

\usepackage{authblk}

\usepackage[a4paper,left=1.0in,right=1.0in,top=3cm,bottom=3cm]{geometry}

\usepackage[T1]{fontenc}
\usepackage{lmodern}
\usepackage{microtype}

\usepackage{fixmath}

\numberwithin{equation}{section}

\newtheorem{theorem}{Theorem}[section]
\newtheorem{lemma}[theorem]{Lemma}
\newtheorem{proposition}[theorem]{Proposition}
\newtheorem{corollary}[theorem]{Corollary}
\theoremstyle{definition}
\newtheorem{definition}[theorem]{Definition}
\theoremstyle{remark}
\newtheorem{remark}[theorem]{Remark}
\newtheorem{ex}[theorem]{Example}

\newcommand{\R}{\mathbb{R}}

\newcommand{\h}{H^1_V(\mathcal{M})}

\usepackage{todonotes}

\newcommand{\M}{\mathcal{M}}
\newcommand{\RR}{\mathbb R}
\newcommand{\NN}{\mathbb N}
\newcommand{\N}{\mathbb N}

\DeclareMathOperator*{\essinf}{ess\, inf}

\numberwithin{equation}{section}

\DeclareSymbolFont{rsfs}{U}{rsfs}{m}{n}
\DeclareSymbolFontAlphabet{\mathscr}{rsfs}

\begin{document}

\title{Elliptic problems on complete non-compact
Riemannian manifolds with asymptotically non-negative Ricci curvature}

\author{Giovanni Molica Bisci\thanks{Email address: \texttt{gmolica@unirc.it}}}
\affil{\small Dipartimento P.A.U., Universit\`a  degli
Studi Mediterranea di Reggio Calabria, Salita Melissari - Feo di
Vito, 89100 Reggio Calabria (Italy)}

\author{Simone Secchi\thanks{Email address: \texttt{Simone.Secchi@unimib.it}}}
\affil{\small Dipartimento di Matematica e Applicazioni Universit\`a degli Studi di Milano-Bicocca, via Roberto Cozzi 55, I-20125, Milano, Italy}

\maketitle

\begin{abstract}
In this paper we discuss the existence and non--existence of weak solutions to parametric equations involving the Laplace--Beltrami operator $\Delta_g$ in a complete non--compact $d$--dimensional ($d\geq 3$) Riemannian manifold $(\mathcal{M},g)$ with asymptotically non--negative Ricci curvature and intrinsic metric $d_g$. Namely, our simple model is the following problem
\begin{equation*}
\left\{
\begin{array}{ll}
-\Delta_gw+V(\sigma)w=\lambda \alpha(\sigma)f(w) & \mbox{ in } \mathcal{M}\\
w\geq 0 & \mbox{ in } \mathcal{M}
\end{array}\right.
\end{equation*}
where $V$ is a positive coercive potential, $\alpha$ is a positive
bounded function, $\lambda$ is a real parameter and $f$ is a suitable
continuous nonlinear term. The existence of at least two non--trivial
bounded weak solutions is established for large value of the parameter
$\lambda$ requiring that the nonlinear term $f$ is non--trivial,
continuous, superlinear at zero and sublinear at infinity. Our
approach is based on variational methods. No assumptions on the
sectional curvature, as well as symmetry theoretical arguments, are
requested in our approach.

\medskip


\end{abstract}

\tableofcontents
\bigskip
\bigskip
\hfill \emph{Dedicated to Professor Carlo Sbordone}\par
 \hfill \emph{on the occasion of his 70th birthday}
\section{Introduction}\label{sec:introduzione}

It is well-known that sign conditions on the Ricci curvature $\operatorname{Ric}_{(\M,g)}$ give topological
and diffeomorphic informations on a Riemannian manifold $(\M,g)$. For instance,
the Myers theorem affirms that a complete Riemannian manifold whose Ricci curvature
satisfies the inequality
\[
\operatorname{Ric}_{(\M,g)} \geq kg,
\]
for some positive constant $k$, is compact. In the same spirit the
Hamilton's theorem ensures that a compact simply connected three
dimensional Riemannian manifold whose Ricci curvature is positive is
diffeomorphic to the three dimensional sphere; see the celebrated
books of Hebey \cite{EH1,EH2} for details.

Among these intriguing geometric implications, conditions on the Ricci
curvature produce meaningful compact Sobolev embeddings of certain
weighted Sobolev space associated to $\M$ into the Lebesgue spaces;
see Lemma \ref{immer} below. The compact embeddings recalled above are
used in proving essential properties of the energy functional
associated to elliptic problems on $\M$, in order to apply the
minimization theorem or the critical point theory. Hence, the
aforementioned compact embedding results give rise to applications to
differential equations in the non--compact framework.

In this order of ideas, this paper is concerned with the existence of
solutions to elliptic problems involving the Laplace--Beltrami
operator $\Delta_g$ on a complete, non--compact $d$--dimensional
($d\geq 3$) Riemannian manifold $(\mathcal{M},g)$ with asymptotically
non--negative Ricci curvature $\operatorname{Ric}_{(\mathcal{M},g)}$
with a base point $\bar{\sigma}_0\in \mathcal{M}$. More precisely, we
assume that the following technical condition on
$\operatorname{Ric}_{(\mathcal{M},g)}$ holds:
\begin{itemize}
\item[$(As_{R}^{H, g})$] there exists a non--negative bounded function $H\in C^1([0,\infty))$ satisfying
\[
\int_0^{+\infty}tH(t)\, dt<+\infty,
\]
such that
\[
\operatorname{Ric}_{(\mathcal{M},g)}(\sigma)\geq -(d-1)H(d_g(\bar{\sigma}_0,\sigma)),
\]
for every $\sigma\in \mathcal{M}$, where  $d_g$ is the distance function associated to the Riemannian metric $g$ on $\mathcal{M}$.
\end{itemize}

Roughly speaking, condition $(As_{R}^{H, g})$ express a control on the Ricci curvature $\operatorname{Ric}_{(\mathcal{M},g)}$ that is governed in any point $\sigma\in M$ by the value
\[
h(\sigma):=-(d-1)H(d_g(\bar{\sigma}_0,\sigma)),
\]
in any direction. See \cite{PRS} for more details concerning the geometrical meaning of the above condition.\par

Let $V\colon \mathcal{M}\rightarrow \R$ be a continuous function satisfying the following conditions:
\begin{itemize}
\item[$(V_1)$] $\nu_0:=\inf_{\sigma\in \mathcal{M}}V(\sigma)>0$;
\item[$(V_2)$] \textit{there exists $\sigma_0\in \mathcal{M}$ such that}
\[
\lim_{d_g(\sigma_0,\sigma)\rightarrow +\infty}V(\sigma)=+\infty.
\]
\end{itemize}
 Motivated by a wide interest in the current literature on elliptic problems on manifolds, under the variational viewpoint, we study
here the existence and non--existence of weak solutions to the following problem
\begin{equation}\label{problema}
\left\{
\begin{array}{ll}
-\Delta_gw+V(\sigma)w=\lambda \alpha(\sigma)f(w) & \mbox{ in } \mathcal{M}\\
w\geq 0 & \mbox{ in } \mathcal{M}\\
w\rightarrow 0 & \mbox{ as } d_g(\sigma_0,\sigma)\rightarrow +\infty,
\end{array}\right.
\end{equation}
where $\lambda$ is a positive real parameter and $\alpha\colon \mathcal{M}\to\R$ is a function belonging to $L^{\infty}(\mathcal{M})\cap L^{1}(\mathcal{M})\setminus\{0\}$ and satisfying
\begin{equation}\label{proprietabeta}
\alpha(\sigma) \geq 0\quad\mbox{for a.e.}\,\, \sigma\in \mathcal{M}.
\end{equation}

\begin{remark}
	We just point out that the potential $V$ has
	a crucial r\^{o}le concerning the existence and behaviour of
	solutions. For instance, after the seminal paper of
	Rabinowitz \cite{seminalrabinowitz}, where the potential $V$
	is assumed to be coercive, several different assumptions are adopted in order to obtain
	existence and multiplicity results. For instance, similar assumptions on $V$ have been
	studied on Euclidean spaces, see \cite{B1,B2,B3,F1,F2,F3}. More recently in \cite[Theorem 1.1]{FaraciFarkas} the authors
	studied a characterization theorem concerning the existence of multiple solutions on complete non--compact Riemannian manifolds under the key assumption of asymptotically non--negative Ricci curvature. The results achieved in the cited paper extend some multiplicity properties recently proved in several different framework (see, for instance, the papers \cite{Anello, MRadu2, Radu3, GMBSe2}) and mainly inspired by the work \cite{Riccerinew}.
\end{remark}

In the first part of the paper,  we suppose that $f\colon [0,+\infty)\to\R$ is continuous, \textit{superlinear at zero}, i.e.
\begin{equation}\label{supf}
\lim_{t\to 0^+}\frac{f(t)}{t}=0, \tag{$f_1$}
\end{equation}
\textit{sublinear at infinity}, i.e.
\begin{equation}\label{condinfinito}
\lim_{t\to +\infty}\frac{f(t)}{t}=0, \tag{$f_2$}
\end{equation}
and such that
\begin{equation}\label{Segno}
\sup_{t>0}F(t)>0, \tag{$f_3$}
\end{equation}
where
\[
F(t):=\int_0^t f(\tau)d\tau,
\]
for any $t\in[0,+\infty)$.\par
 Assumptions \eqref{supf} and \eqref{condinfinito} are quite standard in presence of subcritical terms; moreover, together with \eqref{Segno}, they assure that the number
\begin{equation}\label{c6}
c_f:=\max_{t> 0}\frac{f(t)}{t}.
\end{equation}
is well--defined and strictly positive. Further, property \eqref{supf} is a sublinear growth condition at infinity on the nonlinearity $f$ which complements the classical Ambrosetti--Rabinowitz assumption. In the Euclidean case, problems involving only sublinear terms have been also studied by
several authors, see, among others, \cite{Cia, Gigio1}. Other multiplicity results
for elliptic eigenvalue problems involving pure concave nonlinearities on bounded
domains can be found in \cite{Per}.

\medskip

We emphasize that in this paper we are interested in equations depending on parameters.
In many mathematical problems deriving from applications the presence of one
parameter is a relevant feature, and the study of the way solutions depend on
parameters is an important topic. Most of the results in this direction were obtained
through bifurcation theory and variational techniques.

In order to state the main (non)existence theorem in presence of a sublinear term at infinity, let us introduce some notation. Let
\[
B_\sigma(1):= \left\{ \zeta\in \mathcal{M}:d_g(\sigma,\zeta)<1 \right\},
\]
be the open unit geodesic ball of center $\sigma\in \mathcal{M}$, and denote by
\[
\operatorname{Vol}_g(B_{\sigma}(1)):=\int_{B_{\sigma}(1)}dv_g
\]
its volume respect to the intrinsic metric $g$, that is the
$d$--dimensional Hausdorff measure $\mathcal{H}_d$ of
$B_\sigma(1)\subset \mathcal{M}$.

The first main result of the paper is an existence theorem for
equations driven by the Laplace--Beltrami operator, as stated below.
\begin{theorem}\label{generalkernel0f}
Let $(\mathcal{M},g)$ be a complete, non--compact, Riemannian manifold of dimension $d \geq 3$
such that condition $(As_{R}^{H, g})$ holds, and assume that
\begin{equation}\label{c6nuovo}
\inf_{\sigma\in \mathcal{M}}\operatorname{Vol}_g(B_{\sigma}(1))>0.
\end{equation}
Furthermore, let $V$ be a potential for which hypotheses $(V_1)$ and $(V_2)$ are valid. In addition, let $\alpha\in L^{\infty}(\mathcal{M})\cap L^1(\M)\setminus\{0\}$ satisfy \eqref{proprietabeta} and $f\colon [0,+\infty)\to\R$ be a continuous function verifying \eqref{supf}--\eqref{Segno}. Then, the following conclusions hold:
\begin{itemize}
\item[$(i)$] Problem \eqref{problema} admits only the trivial solution whenever
\[
0\leq \lambda <\frac{\nu_0}{c_f\left\|\alpha\right\|_{\infty}} ;
\]
\item[$(ii)$] there exists $\lambda^\star>0$ such that problem \eqref{problema} admits at least two distinct and non--trivial weak solutions $w_{1,\lambda}, w_{2,\lambda}\in L^{\infty}(\mathcal{M}) \cap H_V^{1}(\mathcal{M})$,
    provided that $\lambda>\lambda^\star$.
\end{itemize}
\end{theorem}
The symbol $H_V^{1}(\mathcal{M})$ denotes the Sobolev space endowed by the norm
\[
\|w\|:=\left(\int_{\mathcal{M}}|\nabla_g w(\sigma)|^2dv_g+\int_{\mathcal{M}}V(\sigma)|w(\sigma)|^2dv_g\right)^{1/2},\quad w\in H_V^{1}(\mathcal{M})
\]
that will be defined in Section \ref{sec:preliminaries}. Furthermore, let $a<b$ be two positive constants, and define the following annular domain
\[
A_a^b(\sigma_0):=\left\{\sigma\in \M:b-a< d_g(\sigma_0,\sigma)< a+b \right\}.
\]
The conclusion of Theorem \ref{generalkernel0f} -- part $(ii)$ is valid for every
    \[
    \lambda>\frac{m_{V}t_0^2\left(1+\frac{1}{(1-\varepsilon_0)^2r^2}\right)\operatorname{Vol}_g(A_r^{\rho}(\sigma_0))}{\alpha_0F(t_0)\operatorname{Vol}_g(A_{\varepsilon_0 r}^{\rho}(\sigma_0))-
 \|\alpha\|_{\infty}\max_{t\in(0,t_0]}
 |F(t)|\operatorname{Vol}_g(A_{r}^{\rho}(\sigma_0)\setminus A_{\varepsilon_0 r}^{\rho}(\sigma_0))},
    \]
    where  $m_{V}:=\max\left\{1,\max_{\sigma\in A_r^{\rho}(\sigma_0)}V(\sigma)\right\}$, $0<r<\rho$ such that
    \[
    \operatorname*{ess\, inf}_{\sigma\in A_{r}^{\rho}}\alpha(\sigma)\geq \alpha_0>0,
    \]
    $t_0\in (0,+\infty)$ and $F(t_0)>0$ with
    \begin{equation*}\label{Fru}
    \frac{\alpha_0F(t_0)}{\|\alpha\|_{\infty}\max_{t\in(0,t_0]}
 |F(t)|}>
 \frac{\operatorname{Vol}_g(A_{r}^{\rho}(\sigma_0)\setminus A_{\varepsilon_0 r}^{\rho}(\sigma_0))}{\operatorname{Vol}_g(A_{\varepsilon_0 r}^{\rho}(\sigma_0))},
     \end{equation*}
    for some $\varepsilon_0\in [1/2,1)$; see Remarks \ref{parametri} and \ref{parametri2} for details.

By a three critical points result stated in \cite{Ricceri} one can prove that the number of solutions
for problem \eqref{problema} is stable under small nonlinear perturbations of subcritical type for every $\lambda>\lambda^{\star}$. More precisely, the statement of Theorem \ref{generalkernel0f} remain valid for the following perturbed problem
\begin{equation*}\label{problemaperturbed}
\left\{
\begin{array}{ll}
-\Delta_gw+V(\sigma)w=\lambda \alpha(\sigma)f(w)+\beta(\sigma)g(w) & \mbox{ in } \mathcal{M}\\
w\geq 0 & \mbox{ in } \mathcal{M}\\
w\rightarrow 0 & \mbox{ as } d_g(\sigma_0,\sigma)\rightarrow +\infty,
\end{array}\right.
\end{equation*}
where the perturbation term $g\colon [0,+\infty)\rightarrow \R$, with $g(0)=0$, satisfies
\[
\sup_{t>0}\frac{|g(t)|}{t+t^{q-1}}<+\infty,
\]
for some $q\in (2,2^*)$, and $\beta\in L^{\infty}(\mathcal{M})\cap L^{1}(\mathcal{M})$ is assumed to be non--negative on $\M$.

Theorem \ref{generalkernel0f} will be proved by adapting variational techniques to the non--compact manifold setting.
More precisely, with some minimization techniques in addition to the Mountain Pass Theorem, we are able to prove the existence of
at least two weak solutions whenever the parameter $\lambda$ is sufficiently large. Furthermore, the boundness of the solutions immediately follows by \cite[Theorem 3.1]{FaraciFarkas}.\par

 The methods used may be suitable for other purposes, too. Indeed,
 we recall that a  variational approach has been extensively
used in several contexts, in order to prove multiplicity results of different
problems, such as elliptic problems on either bounded or unbounded domains of the Euclidean space (see \cite{k0,k2,k3,k5}), elliptic equations involving the Laplace--Beltrami operator on compact Riemannian manifold without boundary (see \cite{k1}), and, more recently, elliptic equations on the ball endowed with Funk--type metrics \cite{k4}. Furthermore in the non--compact setting, Theorem \ref{generalkernel0f} has been proved for elliptic problems on Cartan--Hadamard manifolds (i.e. simply connected, complete Riemannian manifolds with non--positive sectional curvature) with poles in \cite[Theorem 4.1]{FFK}
and, in presence of symmetries, for Schr\"{o}dinger--Maxwell systems on Hadamard
manifolds in \cite[Theorem 1.3]{FK}.

More precisely, in \cite[Theorem 4.1]{FFK} the authors study the  Schr\"{o}dinger-type
equation
\begin{equation}\label{CH-equation}
-\Delta_gw+V(\sigma)w=\lambda\frac{s^2_{k_0}(d_{12}/2)}{d_1d_2s_{k_0}(d_1)s_{k_0}(d_2)}+\mu W(\sigma)f(u)
\end{equation}
on a Cartan--Hadamard manifold $(\M,g)$ with sectional curvature
$\mathbf{K} \geq \kappa_0$ (for some $\kappa_0\leq 0$) and two poles $S:=\{x_1,x_2\}$, where the function $s_{k_0}\colon [0,+\infty)\rightarrow \R$ is defined by
 \begin{equation*}
 s_{k_0}(r):= \left\{
 \mkern-8mu
\begin{array}{ll}
\frac{\sinh(r\sqrt{-k_0})}{r} & \mbox{ if } \kappa_0<0\\
r & \mbox{ if } \kappa_0=0,
\end{array}\right.
\end{equation*}
 and
 \begin{equation*}
 d_{12}:=d_g(x_1,x_2),\quad d_{12}:=d_g(x_1,x_2),\quad d_i=d_g(\cdot,x_i),\,\, i=1,2.
 \end{equation*}
 Furthermore, the potentials $V$, $W\colon \M\rightarrow\R$ are
 positive, $f\colon [0,\infty)\rightarrow\R$ is a suitable continuous
 function sublinear at infinity, $\lambda\in [0,(n-2)^2)$ and
 $\mu\geq 0$.  An analog of Theorem \ref{generalkernel0f} for Problem
 \eqref{CH-equation} has been proved in \cite[Theorem 4.1]{FFK} by
 means of a suitable compact embedding result (see \cite[Lemma
 5]{FFK}). The proof of this lemma is based on an interpolation
 inequality and on the Sobolev inequality valid on Cartan--Hadamard
 manifolds (see \cite[Chapter 8]{EH1}).

 Again on Hadamard manifolds of dimension $3\leq d\leq 5$,  an interesting counterpart of Theorem \ref{generalkernel0f} has been proved in \cite[Theorem 1.3]{FK} for Schr\"{o}dinger--Maxwell systems of the form
\begin{equation*}\label{problemasystem}
\left\{\mkern-8mu
\begin{array}{ll}
-\Delta_gw+w+e w\phi=\lambda \alpha(\sigma)f(w) & \mbox{ in } \mathcal{M}\\
-\Delta_g\phi+\phi=q w^2 & \mbox{ in } \mathcal{M}.
\end{array}\right.
\end{equation*}
A key tool used along this paper is the existence of a suitable topological group $G$
acting on the Sobolev space $H^1_g(\M)$, such that the $G$--invariant closed subspace $H_{G,T}^1(\M)$ can be compactly embedded in suitable Lebesgue spaces. The classical Palais criticality principle (see \cite{pala}) produces the multiplicity result.

\bigskip

On the contrary of the cited cases, we do not require here any assumption on the sectional curvature of the ambient manifold. Moreover, no upper restriction on the topological dimension $d$ is necessary, as well as we do not use $G$--invariance arguments along our proof. In the main approach the compact embedding of the Sobolev space $\h$ into Lebesgue spaces $L^{\nu}(\M)$, with $\nu\in (2,2^*)$, is obtained thanks to a Rabinowitz--type result proved recently in \cite[Lemma 2.1]{FaraciFarkas}. The structural hypotheses on $V$ and the curvature hypothesis $(As_{R}^{H, g})$ are sufficient conditions to ensure this compactness embedding property. For the sake of completeness we recall \cite[Lemma 2.1]{FaraciFarkas} in Lemma \ref{immer} below.

 We also mention the recent paper \cite{MRadu1}, where Theorem \ref{generalkernel0f} is exploited when the underlying operator is of (fractional) nonlocal type.\par

\bigskip

 In the Euclidean setting, a meaningful version of Theorem \ref{generalkernel0f} can be done requiring different (weaker) assumptions on the potential $V$ that imply again the compactness of the
embedding of $H^{1}_V(\R^d)$ into the Lebesgue spaces $L^{\nu}(\M)$, with $\nu\in (2,2^*)$ (see \cite{FaraciFarkas} and Remark \ref{immergo}). More precisely, the result reads as follows.

\begin{theorem}\label{boccia}
  Let $\R^d$ be the $d$--dimensional $(d\geq 3)$ Euclidean space and
  let $V\in C^0(\R^d)$ be a potential such that
  $v_0=\inf_{x\in \R^d}V(x)>0$, and
\[
\lim_{|x|\rightarrow +\infty}\int_{B(x)}\frac{dy}{V(y)}=0,
\]
where $B(x)$ denotes the unit ball of center $x\in \R^{d}$.\par
In addition, let
$\alpha\in L^{\infty}(\mathcal{M})\cap L^1(\M)\setminus\{0\}$ satisfy
\eqref{proprietabeta} and let $f\colon [0,+\infty)\to\R$ be a
continuous function verifying \eqref{supf}--\eqref{Segno}. Consider
the following problem
\begin{equation}\label{problema3}
\left\{
\begin{array}{ll}
-\Delta u+V(x)u=\lambda \alpha(x) f(u) & \mbox{ in } \R^d\\
u\geq 0 & \mbox{ in } \R^d\\
u\rightarrow 0 & \mbox{ as } |x|\rightarrow +\infty.
\end{array}\right.
\end{equation}
Then, the following conclusions hold:
\begin{itemize}
\item[$(i)$] Problem \eqref{problema3} admits only the trivial solution whenever
\[
0\leq \lambda <\frac{v_0}{c_f} ;
\]
\item[$(ii)$] there exists $\lambda^\star_E>0$ such that \eqref{problema3} admits at least two distinct and non--trivial weak solutions $u_{1,\lambda}$,  $u_{2,\lambda}\in L^{\infty}(\R^d) \cap H_V^{1}(\R^d)$,
    provided that $\lambda>\lambda^\star_E$.
\end{itemize}
\end{theorem}
We notice that $\lambda^\star_E\leq \lambda^{0}_E$, where
\begin{equation*}
\lambda^{0}_E:=\frac{m_{V}t_0^2\left(1+\displaystyle\frac{1}{(1-\varepsilon_0)^2r^2}\right) \left((r+\rho)^d-(\rho-r)^d \right)}{\lambda_D(t_0,\varepsilon_0,r,\rho,\alpha,f,d)},
\end{equation*}
and
\begin{multline*}
\lambda_D(t_0,\varepsilon_0,r,\rho,\alpha,f,d):=
\alpha_0 F(t_0) \left( \left(\varepsilon_0r+\rho \right)^d- \left(\rho-\varepsilon_0 r \right)^d \right) \\
{}-\|\alpha\|_{\infty} \max_{t\in(0,t_0]}
 |F(t)| \left[ \left((r+\rho)^d-(\rho-r)^d) \right)-\left((\varepsilon_0r+\rho)^d-(\rho-\varepsilon_0r)^d \right) \right],
\end{multline*}
see Remark \ref{parametri3} for details.

Of course if $V$ is a strictly positive continuous and coercive function, the above conditions hold true.
A different form of Corollary \ref{boccia} has been proved in \cite[Theorem 1.1]{KRO} requiring the radial symmetry of the weights $V$ and $\alpha$. A careful analysis of the natural isometric action of the orthogonal group $O(d)$ on the Sobolev space $H_V^{1}(\R^d)$ is a meaningful and crucial argument used by the authors of
the cited paper. We emphasize that in Corollary \ref{boccia}, on the contrary of \cite[Theorem 1.1]{KRO}, we do not require any symmetry assumption on the coefficients $V$ and $\alpha$. On the other hand, we point out that Corollary \ref{boccia} can be also obtained as a direct application of \cite[Theorem 4.1]{FFK} modeled to the Euclidean flat case.\par

 Requiring the validity of a special case of the celebrated Ambrosetti--Rabinowitz growth condition (see relation \eqref{eq:mu0} below), a complementary result (respect to Theorem \ref{generalkernel0f}) is valid for superlinear problems at infinity. More precisely, it follows that

\begin{theorem}\label{generalkernel0f222}
Let $(\mathcal{M},g)$ be a complete, non--compact, Riemannian manifold of dimension $d \geq 3$
such that conditions $(As_{R}^{H, g})$ and \eqref{c6nuovo} hold. Let $V$ be a potential for which hypotheses $(V_1)$ and $(V_2)$ are valid, $\alpha\in L^{\infty}(\mathcal{M})_{+}\cap L^{1}(\mathcal{M})\setminus\{0\}$ and let $f\colon [0,+\infty)\to\R$ be a continuous function verifying condition \eqref{supf},
\begin{equation}\label{supg}
\sup_{t> 0}\frac{|f(t)|}{t+t^{q-1}}<+\infty,
\end{equation}
for some $q\in (2,2^*)$, and
\begin{equation}\label{eq:mu0}
\mbox{there exist $\nu>2$ such that for every $t>0$ one has $0<\nu F(t)\le tf(t)$.} \tag{AR}
\end{equation}
Then, for every $\lambda>0$, problem \eqref{problema}
admits at least one non--trivial weak solution $w\in L^{\infty}(\mathcal{M}) \cap H_V^{1}(\mathcal{M})$.
\end{theorem}

The r\^{o}le of condition \eqref{eq:mu0} is to ensure the boundness of Palais--Smale sequences of the
Euler--Lagrange functional associated to Problem \eqref{problema}. This is crucial in the applications of critical point theory. However, although \eqref{eq:mu0} is a quite natural condition, it is somewhat restrictive and eliminates many nonlinearities, as for instance the nonlinear terms for which Theorem \ref{generalkernel0f} holds.

\bigskip

The last part of the paper is dedicated to subcritical problems with asymptotic superlinear behaviour at zero.
More precisely, we prove that there exists a well--localized interval of positive real
parameters $(0,\lambda_{\star})$
such that, for every $\lambda\in (0,\lambda^{\star})$,
problem \eqref{problema} admits at least one non--trivial weak solution in a suitable Sobolev space $\h$.

Set
\begin{equation*}
c_\ell:=\sup_{u \in H_V^1(\mathcal{M}) \setminus \{0\}} \frac{\|u\|_\ell}{\|u\|},
\end{equation*}
for every $\ell\in (2,2^*)$.\par
\smallskip
With the above notations the main result reads as follows.
\begin{theorem}\label{Main1}
Let $(\mathcal{M},g)$ be a complete, non--compact, Riemannian manifold of dimension $d \geq 3$
such that conditions $(As_{R}^{H, g})$ and \eqref{c6nuovo} hold.
Let $f\colon [0,+\infty)\to{\R}$ be a continuous function with $f(0)=0$ and satisfying the growth condition
\begin{equation}\label{crescita2}
\beta_f:=\sup_{t> 0}\frac{|f(t)|}{1+t^{q-1}}<+\infty,
\end{equation}
for some $q\in (2,2^*)$ as well as
\begin{equation}\label{azero}
-\infty < \liminf_{t \to 0^+} \frac{F(t)}{t^2} \leq \limsup_{t \to 0^+} \frac{F(t)}{t^2} < +\infty,
\end{equation}
where $F(t):=\displaystyle\int_0^{t}f(\tau)d\tau$.

Furthermore, let $\alpha\in L^{\infty}(\M)\cap L^{p}(\M)\setminus\{0\}$, where $p:=q/(q-1)$ is the conjugate Sobolev exponent of $q$, and such that \eqref{proprietabeta} holds. Finally, let $V$ satisfy $(V_1)$ and $(V_2)$.

Under these assumptions, there exists a positive number
\begin{equation}\label{lambda1}
\lambda_{\star}:=\frac{1}{2(q-1)\beta_fc_q^2}\left(\frac{q(q-2)^{q-2}}{\|\alpha\|_p^{q-2}\|\alpha\|_{\infty}}\right)^{1/(q-1)},
\end{equation}
such that, for every $\lambda\in (0,\lambda_{\star})$,
the problem
\begin{equation*}\label{problemanuovooo}
\left\{
\begin{array}{ll}
-\Delta_gw+V(\sigma)w=\lambda \alpha(\sigma)f(w) & \mbox{ in } \mathcal{M}\\
w\geq 0 & \mbox{ in } \mathcal{M}
\end{array}\right.
\end{equation*}
 \noindent admits at least one non--trivial weak solution $w_{\lambda}\in \h$ such that
\[
\lim_{\lambda\rightarrow 0^+}\|w_{\lambda}\|=0.
\]
\end{theorem}

Condition \eqref{azero} is not new in the literature and it has been used in order to study existence
and multiplicity results for some classes of elliptic problems on bounded domains of the Euclidean
space: see, among others, the papers \cite{kagi1, kagi3, kagi2} and \cite{GMBSe}. An application
to Schr\"{o}dinger equations in presence of either radial or axial symmetry, again on the classical
Euclidean space, has been recently proposed in \cite[Theorem 6]{GMBSe}. To the best of our knowledge no further applications in a non--compact framework have been achieved requiring this hypothesis. We also emphasize that Theorem \ref{Main1} can be proved without any use of the Ambrosetti--Rabinowitz condition.

\medskip

The existence of weak solutions of the following problem
\begin{equation}\label{Dirichletsh}
\left\{
\begin{array}{l}
-\Delta u+V(x)u= \lambda \alpha(x)f(u) \,\,\,\,\textrm{ in}\,\,\,\R^d\\
\smallskip
u\in H^{1}(\R^d),
\end{array}
\right.
\end{equation}
or
some of its variants, has been studied after the paper \cite{bw}; see, for instance, the
book \cite{Wi} and papers \cite{Kri, KRO}. We also cite the paper \cite{ClappWeth}, and references
therein, where the authors prove the existence of multiple weak solutions for suitable Schr\"{o}dinger
equations  under some weak one-sided asymptotic estimates on the terms $V$ and $\alpha$
(without symmetry assumptions); see, for the sake of completeness, the papers \cite{BCW(MatANN),GW}.

As explicitly claimed in \cite{bw,bw2}, an interesting prototype for Problem \eqref{Dirichletsh} is given by
\begin{equation}\label{Cc}
\left\{
\begin{array}{l}
-\Delta u+V(x)u= \lambda \alpha(x)(|u|^{r-2}u+|u|^{s-2}u)  \quad \mbox{in $\mathbb{R}^d$}\\
\smallskip
u\in H^{1}(\R^d),
\end{array}
\right.
\end{equation}
where $1<r<2< s<2d/(d-2)$, the potential $V$ is bounded from below by a positive constant and
$\alpha \in L^\infty(\mathbb{R}^d)$.

In this setting, a simple case of Theorem \ref{Main1} reads as follows.
\begin{corollary}\label{Main3}
Let $(\mathcal{M},g)$ be a complete, non-compact, Riemannian manifold of dimension $d \geq 3$
such that conditions $(As_{R}^{H, g})$ and \eqref{c6nuovo} hold.
Furthermore, let \(1<r<2< s<2^{*}\)
and let $\alpha\in L^{\infty}(\M)\cap L^{\frac{s}{s-1}}(\M)\setminus\{0\}$ be a non--negative map.

Then, for $\lambda$ sufficiently small, the following problem
\begin{equation}\label{problemaspeciale}
\left\{
\begin{array}{ll}
-\Delta_gw+V(\sigma)w=\lambda \alpha(\sigma)(u^{r-1}+u^{s-1}) & \mbox{ in } \mathcal{M}\\
w\geq 0 & \mbox{ in } \mathcal{M},
\end{array}\right.
\end{equation}
 \noindent admits at least one non--trivial weak solution $w_{\lambda}\in \h$ such that
\[
\lim_{\lambda\rightarrow 0^+}\|w_{\lambda}\|=0.
\]
\end{corollary}
We notice that Corollary \ref{Main3} is  new even in the Euclidean case. For instance, it is easily seen that Theorem 1.1 in \cite{KRO} cannot be applied to the Euclidean counterpart of Problem \eqref{problemaspeciale}. Thus, as a byproduct of Corollary \ref{Main3}, the existence of a non--trivial weak solution to Problem \eqref{Cc} is obtained provided that $\lambda$ is sufficiently small and without any symmetry assumptions on the coefficients. For more results on subelliptic eigenvalue problems on unbounded domains of stratified Lie groups we refer to \cite{MP1, Pucci}.

\bigskip

The paper is structured as follows. After introducing the functional space related to problem
\eqref{problema} together with its basic properties (Section \ref{sec:preliminaries}), we show
through direct computations that for a determined right neighborhood of $\lambda$,
the zero solution is the unique one (Subsection \ref{sec:non-exist}).
In Subsection \ref{sec:multiple} we prove the existence of two weak solutions for
$\lambda$ bigger than some $\lambda^\star$: the first one obtained via direct
minimization, the second via the Mountain Pass Theorem. We refer to the book \cite{KRV} for the abstract variational setting used along the present paper. See also the recent contribution \cite{BMPR} for related topics.

\section{Abstract framework}\label{sec:preliminaries}
In this subsection we briefly recall the definitions of the functional space setting; see \cite{EH1} and \cite{EH2} for more details.

\subsection{Sobolev spaces on Manifolds}

Let $(\M,g)$ be a $d$--dimensional ($d\geq 3$) Riemannian manifold,
and let $g_{ij}$ be the components of the $(2,0)$-metric tensor $g$.
We denote by $T_\sigma \M$ the tangent space at
$\sigma\in \mathcal{M}$ and by
$T\M:= \bigcup_{\sigma\in \mathcal{M}} T_\sigma \M$ the tangent bundle
associated to $\mathcal{M}$. In terms of local coordinates,
\[
g=g_{ij}dx^{i}\otimes dx^{j},
\]
where
$dx^{i}\otimes dx^{j}\colon T_{\sigma}\M\times
T_{\sigma}\M\rightarrow\mathbb{R}$ is the quadratic form defined by
\[
dx^{i}\otimes dx^{j}(u,v)=dx^{i}(u)dx^{j}(v), \qquad\forall\, u,v\in T_{\sigma}\M.
\]
Note that, here and in the sequel, Einstein's summation convention is
adopted.

Let $d_g \colon \mathcal{M}\times\mathcal{M} \rightarrow [0,+\infty)$ be the natural distance function associated to the Riemannian metric $g$ and let $B_\sigma(r) := \{\zeta\in\mathcal{M} : d_g(\sigma,\zeta) <r\}$ be the open geodesic ball with center $\sigma$ and radius $r>0$. If $C^\infty(\M)$ denotes as usual the space of smooth functions defined on $\M$, we set
\begin{equation}\label{normaHduealpha}
\left\|w\right\|_{g}:=\left(\int_{\M}  |\nabla_g w(\sigma)|^2 dv_g+ \int_{\M} |w(\sigma)|^2  dv_g\right)^{1/2},
\end{equation}
for every $w\in C^\infty(\M)$, where $\nabla w$ is the covariant derivative of $w$ and $dv_g$ is the Riemannian measure on $\M$.
The Riemannian volume element $dv_g$ in \eqref{normaHduealpha} is given by
\[
dv_g:=\sqrt{\det(g_{ij})}dx_1\wedge \ldots \wedge dx_d,
\]
where $dx=dx_1\wedge \ldots \wedge dx_d$ stands for the Lebesgue's volume element of $\R^d$, and, if $D\subset \M$ is an open bounded set, we denote by
\[
\operatorname{Vol}_g(D):=\int_{\M}dv_g,
\]
is the $d$--dimensional Hausdorff measure of $D$
with respect to the metric $d_g$. Moreover, it is well--defied the Radon measure $f\mapsto \displaystyle\int_{\M}fdv_g$.

For every $\sigma\in\M$ one has the eikonal equation
\begin{equation}\label{eikonal}
|\nabla_g d_g(\sigma_0,\cdot)|=1\,\quad \mbox{a.e. in }  \M\setminus\{\sigma_0\}
\end{equation}
and in local coordinates $(x^1,\ldots, x^d)$, $\nabla_g w$ can be represented by
\[
(\nabla^2_g w)_{ij}=\frac{\partial^2 w}{\partial x^i \partial x^j} - \Gamma_{ij}^k \frac{\partial w}{\partial x^k}
\]
where
\[
\Gamma_{ij}^k:=\frac{1}{2}\left(\frac{\partial g_{lj}}{\partial x^i} + \frac{\partial g_{li}}{\partial x^j} - \frac{\partial g_{ij}}{\partial x^k}\right) g^{lk}
\]
are the usual Christoffel's symbols and $g^{lk}$ are the elements of the inverse matrix of $g$.
Furthermore, in a local neighborhood, the Laplace--Beltrami operator is the differential operator defined by
\[
\Delta_g w := g^{ij}\left(\frac{\partial^2 w}{\partial x^i\partial x^j} -\Gamma_{ij}^k\frac{\partial w}{\partial x^k} \right)=-\frac{1}{\sqrt{\textrm{det}(g_{ij})}}\frac{\partial}{\partial x^{m}}\left(\sqrt{\textrm{det}(g_{ij})}g^{km}\frac{\partial w}{\partial x^k}\right).
\]
\indent Now, as is well-known, the notions of curvatures on a manifold
$(\M,g)$ are described by the {Riemann tensor} $R_{(\M,g)}$ that, for
each point $\sigma\in \M$, gives a multilinear function
\begin{equation*}
  R_{(\M,g)}\colon T_\sigma\M\times T_\sigma\M\times T_\sigma\M\times
  T_\sigma\M\rightarrow \R
 \end{equation*}
such that the symmetries conditions
\[
R_{(\M,g)}(u,v,y,z)=-R_{(\M,g)}(v,u,z,y)=R_{(\M,g)}(z,y,v,u),
\]
and the Bianchi identity
\[
R_{(\M,g)}(u,v,y,z)+R_{(\M,g)}(v,y,u,z)+R_{(\M,g)}(y,u,v,z)=0,
\]
are verified, for every $u,v,y,z\in T_\sigma\M$.

More precisely, $R_{(\M,g)}$ is the $(1,3)$-tensor locally given by
\[
R_{ijk}^{h}=\frac{\partial \Gamma_{jk}^{h}}{\partial x^i}-\frac{\partial \Gamma_{ik}^{h}}{\partial x^j}+\Gamma^{r}_{jk}\Gamma^{h}_{ir}-\Gamma^{r}_{ik}\Gamma^{h}_{jr}.
\]
We can define the {Ricci tensor} $\operatorname{Ric}_{(\M,g)}$ on a
manifold $\M$ as the trace of the Riemann curvature tensor
$R_{(\M,g)}$. Consequently, in local coordinates,
$\operatorname{Ric}_{(\M,g)}$ has the form
\[
R_{ik}=R_{ihk}^{h}=\frac{\partial \Gamma_{hk}^{h}}{\partial x^i}-\frac{\partial \Gamma_{ik}^{h}}{\partial x^h}+\Gamma^{r}_{hk}\Gamma^{h}_{ir}-\Gamma^{r}_{ik}\Gamma^{h}_{hr}.
\]
Furthermore, the quadratic form $\operatorname{Ric}_{(\M,g)}$
associated to the Ricci tensor $\operatorname{Ric}_{(\M,g)}$ is said to be the
{Ricci curvature} of the manifold $(\M,g)$.

We say that $\M$ has Ricci curvature $\operatorname{Ric}_{(\mathcal{M},g)}$ bounded from below if there exists $h\in \R$ such that
\[
\operatorname{Ric}_{(\M,g)}(v,v)\geq h\left\langle v,v \right\rangle_g,
\]
for every $(v,\sigma)\in T_{\sigma}\M\times\M$. A simple form of our results can be obtained when $\operatorname{Ric}_{(\M,g)}\geq 0$. In such a case it suffices to choose the function $H$ identically zero on the curvature condition $(As_{R}^{H, g})$.

Although the curvature of a higher dimensional Riemannian manifold is
a much more sophisticated object than the Gauss curvature of a
surface, it is possible to describe a similar concept in terms a two
dimensional plane $\pi\subset T_\sigma\M$. In this spirit, given any
point $\sigma\in \M$ and any two dimensional plane
$\pi\subset T_\sigma\M$, the sectional curvature of $\pi$ is defined
as follows
\[
\kappa(\pi):=\frac{R(u,v,u,v)}{|u|_g^2|v|^2_g-\left\langle u,v \right\rangle_g^2},
\]
where $\{u,v\}$ is a basis of $\pi$. Since the above definition is independent of the choice of
the vectors $\{u,v\}$, we can compute $\kappa(\pi)$ by working with an orthonormal basis of $\pi$;
see the classical book \cite{RG} for details.\par

A smooth curve $\gamma$ is said to
be a geodesic if
\[
\nabla_{\frac{d\gamma}{dt}}\left(\frac{d\gamma}{dt}\right)=0,
\]
 where $\nabla$ is the Levi--Civita connection on $(\M,g)$. In local coordinates, this means that
\[
\frac{d^2(x^k\circ \gamma(t))}{dt^2}+\Gamma^{k}_{ij}\frac{\partial (x^k\circ \gamma(t))}{\partial x^i}\frac{\partial (x^k\circ \gamma(t))}{\partial x^j}=0.
\]
The Hopf--Rinow theorem ensures that any geodesic on a complete Riemannian manifold
$(\M,g)$ is defined on the whole real line.
Given a complete Riemannian manifold $(\M, g)$ and a point $\sigma\in \M$, the injectivity
radius $\operatorname{inj}_{(\M,g)}(\sigma)$ is defined as the largest $r>0$ for which any geodesic $\gamma$
of length less than $r$ and having $\sigma\in \M$ as an endpoint is minimizing. One has that
$\operatorname{inj}_{(\M,g)}(\sigma)>0$, for any $\sigma\in \M$.

The injectivity radius of $(\M, g)$ is then defined as
\[
\operatorname{inj}_{(\M,g)}:=\inf \left\{\operatorname{inj}_{(\M,g)}(\sigma):\sigma\in\M \right\}\geq 0.
\]
We notice that one gets bounds on the
components of the metric tensor from bounds on the curvature and the injectivity
radius; see \cite{EH1}.

Let
\[
H_V^{1}(\mathcal{M}):=\left\{w\in H^1_g(\mathcal{M}):\int_{\mathcal{M}}|\nabla_g w(\sigma)|^2dv_g+\int_{\mathcal{M}}V(\sigma)|w(\sigma)|^2dv_g<+\infty\right\}
\]
\noindent where the Sobolev space $H^{1}_g(\M)$ is defined as the completion of $C^\infty_0(\M)$ with respect to the norm \eqref{normaHduealpha}. Clearly $H_V^{1}(\mathcal{M})$ is a Hilbert space endowed by the inner product
\[
\langle w_1, w_2\rangle:=\int_{\mathcal{M}}\left\langle \nabla_g w_1(\sigma),\nabla_g w_2(\sigma) \right\rangle_g dv_g +\int_{\mathcal{M}}V(\sigma)w_1(\sigma)w_2(\sigma)dv_g
\]
for all $w_1, w_2\in H_V^{1}(\mathcal{M})$ and the induced norm
\[
\|w\|:=\left(\int_{\mathcal{M}}|\nabla_g w(\sigma)|^2dv_g+\int_{\mathcal{M}}V(\sigma)|w(\sigma)|^2dv_g\right)^{1/2},
\]
for every $w\in H_V^{1}(\mathcal{M})$.

Finally, we introduce a class of functions belonging to $H_V^{1}(\mathcal{M})$ which will be useful to prove our main results.
Since $\alpha\in L^{\infty}(\M)\setminus\{0\}$ is a function with $\alpha\geq 0$, one can find two real numbers $\rho>r>0$ and $\alpha_0>0$ such that
\begin{equation}\label{ess}
\essinf_{\sigma\in A_{r}^{\rho}}\alpha(\sigma)\geq \alpha_0>0.
\end{equation}
\indent Hence, fix $\varepsilon\in [1/2,1)$, let $0<r<\rho$ such that \eqref{ess} holds and set $w_{\rho,r}^{\varepsilon}\in \h$ given by
\begin{equation} \label{TestFunction}
w_{\rho,r}^{\varepsilon}(\sigma):=\left\{
\begin{array}{ll}
0 & \mbox{ if $\sigma \in \M \setminus A_r^{\rho}(\sigma_0)$} \\[5pt]
1 & \mbox{ if $\sigma \in A_{\varepsilon r}^{\rho}(\sigma_0)$}\\[5pt]
 \dfrac{r- |d_g(\sigma_0,\sigma)-\rho|}{(1-\varepsilon)r}
& \mbox{ if $\sigma \in A_{r}^{\rho}(\sigma_0)\setminus A_{\varepsilon r}^{\rho}(\sigma_0)$},
\end{array}
\right.
\end{equation}
\noindent for every $\sigma\in \M$.

With the above notation, we clearly have:
 \begin{itemize}
 \item[$i_1)$] $\operatorname{supp}(w_{\rho,r}^{\varepsilon})\subseteq A_r^{\rho}(\sigma_0)$;
 \item[$i_2)$] $\|w_{\rho,r}^{\varepsilon}\|_\infty\leq 1$;
 \item[$i_3)$] $w_{\rho,r}^{\varepsilon}(\sigma)=1$ for every $\sigma\in A_{\varepsilon r}^{\rho}(\sigma_0)$.
 \end{itemize}
 Moreover, a direct computation shows that the following inequality holds
 \begin{equation}\label{stimazz}
 \|w_{\rho,r}^{\varepsilon}\|^2 \leq m_{V}\left(1+\frac{1}{(1-\varepsilon)^2r^2}\right)\operatorname{Vol}_g(A_r^{\rho}(\sigma_0)),
 \end{equation}
 where
 \[
 m_{V}:=\max\left\{1,\inf_{\sigma\in A_r^{\rho}(\sigma_0)}V(\sigma)\right\}.
 \]
 Indeed, one has
 \begin{equation}\label{stimazz22}
 \|w_{\rho,r}^{\varepsilon}\|^2 \leq m_{V}\left(\int_{\mathcal{M}}|\nabla_g w_{\rho,r}^{\varepsilon}(\sigma)|^2dv_g+\int_{\mathcal{M}}|w_{\rho,r}^{\varepsilon}(\sigma)|^2dv_g\right).
 \end{equation}
 Now, thanks to $i_1)-i_3)$ and bearing in mind the eikonal equation
 \eqref{eikonal}, setting
\begin{equation*}
 I:=\int_{\mathcal{M}}|\nabla_g w_{\rho,r}^{\varepsilon}(\sigma)|^2dv_g+\int_{\mathcal{M}}|w_{\rho,r}^{\varepsilon}(\sigma)|^2dv_g,
\end{equation*}
we have that
\begin{align*}
  I&= \int_{A_r^{\rho}(\sigma_0)}|\nabla_gw_{\rho,r}^{\varepsilon}(\sigma)|^2dv_g+\int_{A_r^{\rho}(\sigma_0)}|w_{\rho,r}^{\varepsilon}(\sigma)|^2dv_g\nonumber\\
   &= \int_{A_{\varepsilon r}^{\rho}(\sigma_0)}|\nabla_g w_{\rho,r}^{\varepsilon}(\sigma)|^2dv_g+\int_{A_{\varepsilon r}^{\rho}(\sigma_0)}|w_{\rho,r}^{\varepsilon}(\sigma)|^2dv_g\nonumber\\
   & \quad+\int_{A_{r}^{\rho}(\sigma_0)\setminus A_{\varepsilon r}^{\rho}(\sigma_0)}|\nabla_g w_{\rho,r}^{\varepsilon}(\sigma)|^2dv_g+\int_{A_{r}^{\rho}(\sigma_0)\setminus A_{\varepsilon r}^{\rho}(\sigma_0)}|w_{\rho,r}^{\varepsilon}(\sigma)|^2dv_g\nonumber\\
   &\leq\operatorname{Vol}_g(A_{r}^{\rho}(\sigma_0))+\frac{1}{(1-\varepsilon)^2r^2}\int_{A_{r}^{\rho}(\sigma_0)\setminus A_{\varepsilon r}^{\rho}(\sigma_0)}|\nabla_g (r- |d_g(\sigma_0,\sigma)-\rho|)|^2dv_g\nonumber\\
   &=\operatorname{Vol}_g(A_{r}^{\rho}(\sigma_0))+\frac{1}{(1-\varepsilon)^2r^2}\int_{A_{r}^{\rho}(\sigma_0)\setminus A_{\varepsilon r}^{\rho}(\sigma_0)}|\nabla_g |d_g(\sigma_0,\sigma)-\rho||^2dv_g\nonumber\\
   &=\operatorname{Vol}_g(A_{r}^{\rho}(\sigma_0))+\frac{1}{(1-\varepsilon)^2r^2}\operatorname{Vol}_g(A_{r}^{\rho}(\sigma_0)\setminus A_{\varepsilon r}^{\rho}(\sigma_0))\nonumber\\
   &\leq \left(1+\frac{1}{(1-\varepsilon)^2r^2}\right)\operatorname{Vol}_g(A_r^{\rho}(\sigma_0)).
 \end{align*}
 Then, the above inequality and \eqref{stimazz22} immediately yields \eqref{stimazz}.
 We notice that the class of functions given by
 \begin{equation} \label{TestFunctionspecial}
w_{\rho,r}^{1/2}(\sigma):=\left\{
\begin{array}{ll}
0 & \mbox{ if $\sigma \in \M \setminus A_r^{\rho}(\sigma_0)$} \\[5pt]
1 & \mbox{ if $\sigma \in A_{r/2}^{\rho}(\sigma_0)$}\\[5pt]
 \frac{2}{r}(r- |d_g(\sigma_0,\sigma)-\rho|)
& \mbox{ if $\sigma \in A_{r}^{\rho}(\sigma_0)\setminus A_{r/2}^{\rho}(\sigma_0)$},
\end{array}
\right.
\end{equation}
 was introduced in \cite{FK}. Of course $w_{\rho,r}^{\varepsilon}$, with $\varepsilon\in [1/2,1)$, is a sort of deformation of $w_{\rho,r}^{1/2}$ whose geometrical shape will be crucial for our goals (see, for instance, Proposition \ref{JDE-F(u-0)0} and the proof of Theorem \ref{Main1}).
\subsection{Embedding results}\label{subsec:b}
For complete manifolds with bounded sectional curvature and positive injectivity radius the embedding $H^1_g(\mathcal{M})\hookrightarrow L^{2^*}(\mathcal{M})$ is continuous, see \cite{Aubin}. The same result holds for manifolds with Ricci curvature bounded from below and positive injectivity radius, see the classical book \cite{EH1}. Moreover, following again \cite{EH1}, we recall that the next embedding result holds for complete, non--compact $d$--dimensional manifolds with Ricci curvature bounded from below and $\inf_{\sigma\in \mathcal{M}}\textrm{ Vol}_g(B_{\sigma}(1))>0$.
\begin{theorem}
 Let $(M,g)$ be a complete, non--compact Riemannian manifold of dimension $d \geq 3$ such that its Ricci curvature $\operatorname{Ric}_{(\mathcal{M},g)}$ is bounded from below and assume that condition \eqref{c6nuovo} holds. Then the embedding $H^1_g(\mathcal{M})\hookrightarrow L^{\nu}(\mathcal{M})$ is continuous for $\nu\in [2,2^*]$, where \(2^*=2d/(d-2)\) is the critical Sobolev exponent.
\end{theorem}
 By using the above result, a sign condition on $V$ and the technical assumption $(As_{R}^{H, g})$ ensure that the following embedding property for the Sobolev space $H^1_V(\mathcal{M})$ into the Lebesgue spaces holds true.

\begin{corollary}
 Let $(M,g)$ be a complete, non--compact Riemannian manifold of dimension $d \geq 3$ satisfying condition $(As_{R}^{H, g})$ and \eqref{c6nuovo}. Furthermore, assume that $(V_1)$ holds. Then the embedding $H^1_V(\mathcal{M})\hookrightarrow L^{\nu}(\mathcal{M})$ is continuous for $\nu\in [2,2^*]$.
\end{corollary}

Finally, in order to employ a variational approach we need a
Rabinowitz--type compactness result proved recently in
\cite{FaraciFarkas} and that will be crucial for our purposes.

\begin{lemma}\label{immer}
  Let $(M,g)$ be a complete, non-compact Riemannian manifold of
  dimension $d \geq 3$ such that
  conditions $(As_{R}^{H, g})$ and \eqref{c6nuovo} holds. Furthermore,
  assume that $(V_1)$ and $(V_2)$ are verified. Then the embedding
  $H^1_V(\mathcal{M})\hookrightarrow L^{\nu}(\mathcal{M})$ is compact
  for $\nu\in [2,2^*)$.
\end{lemma}
Thus, if $\nu\in [2,2^{*}]$, there exists a positive constant $c_\nu$
such that
\begin{equation}\label{constants}
  \left(\int_{\M}|w(\sigma)|^{\nu}dv_g\right)^{1/\nu}\leq c_\nu \left(\int_{\mathcal{M}}|\nabla_g w(\sigma)|^2dv_g+\int_{\mathcal{M}}V(\sigma)|w(\sigma)|^2dv_g\right)^{1/2},
\end{equation}
for every $w\in H_V^{1}(\mathcal{M})$. From now on, for every
$\nu\in [2,2^*]$, we denote by
\[
\|w\|_\nu:=\left(\int_{\M}|w(\sigma)|^{\nu}dv_g\right)^{1/\nu}
\]
denotes the usual norm of the Lebesgue space $L^{\nu}(\M)$.

\begin{remark}\label{immergo}
  We point out that the embedding result stated in Lemma \ref{immer}
  is still true requiring that the potential $V$ is measurable
  (instead of continuous) and $(V_1)$ and $(V_2)$ are verified. The
  same conclusion hold, in the Euclidean case, if
  $V\in L^{\infty}_{\mathrm{loc}}(\R^d)_+$ and
\[
\lim_{|x|\rightarrow +\infty}\int_{B(x)}\frac{dy}{V(y)}=0,
\]
where $B(x)$ denotes the unit ball of center $x\in \R^{d}$. See \cite{FaraciFarkas} for some details.
\end{remark}

\subsection{Weak solutions}
\begin{definition}
	Assume that $f\colon \R\rightarrow\R$ is a subcritical function and $\lambda>0$ is fixed. We say that a function $w\in \h$ is a \textit{weak solution} to problem \eqref{problema} if
	\begin{equation}\label{EL}
	\int_{\mathcal{M}}\left\langle \nabla_g w(\sigma),\nabla_g \varphi(\sigma) \right\rangle_g dv_g +\int_{\mathcal{M}}V(\sigma)w(\sigma)\varphi(\sigma)dv_g=\lambda\int_{\M}\alpha(\sigma)f(w(\sigma))\varphi(\sigma)dv_g,
	\end{equation}
	for every $\varphi\in\h$.
\end{definition}
By direct computation, equation \eqref{EL} represents the variational formulation of \eqref{problema} and the energy functional $J_{\lambda}:\h\to \R$ associated with \eqref{problema} is defined by
\begin{equation*}\label{JKlambda}
 J_{\lambda}(w) :=\frac{1}{2}\left(\int_{\mathcal{M}}|\nabla_g w(\sigma)|^2dv_g+\int_{\mathcal{M}}V(\sigma)|w(\sigma)|^2dv_g\right)
-\lambda\int_{\M}\alpha F(w(\sigma))\, dv_g,
\end{equation*}
for every $w\in \h$.\par
Indeed, as it easily seen, under our assumptions on the nonlinear term, the functional
$J_{\lambda}$ is well defined and of class $C^1$ in $\h$. Moreover, its critical points
are exactly the weak solutions of the problem \eqref{problema}.\par
For a function $f\colon [0,+\infty)\rightarrow \R$ with $f(0)=0$, we set
\[
F_+(t):=\int_0^t  f_+(\tau)d\tau,
\]
for every $t\in \R$, where
\begin{equation}\label{fpiu}
f_+(\tau):=
\left\lbrace
\begin{array}{ll}
f(\tau) & \mbox{ if } \tau\geq 0\\
0 & \mbox{ if } \tau< 0.
\end{array}
\right.
\end{equation}
Furthermore, let $\mathcal J_{\lambda}\colon \h\to \R$ be the functional given by
\begin{equation}\label{JKlambda+}
\begin{aligned}
\mathcal J_{\lambda}(w)& :=\frac{1}{2}\left(\int_{\mathcal{M}}|\nabla_g w(\sigma)|^2dv_g+\int_{\mathcal{M}}V|w(\sigma)|^2dv_g\right)
-\lambda\int_{\M}\alpha(\sigma) F_+(w(\sigma))\, dv_g.
\end{aligned}
\end{equation}

 Our approach to prove existence and multiplicity results to Problem \eqref{problema} consists of applying variational methods to the functional $\mathcal J_{\lambda}$. To this end, we write $\mathcal J_{\lambda}$ as
\[
\mathcal J_{\lambda}(w)= \Phi(w)-\lambda \Psi(w),
\]
where
\begin{equation}\label{psis}
\Phi(w):=\frac 1 2 \left\|w\right\|^2,
\end{equation}
while
\[
\Psi(w):= \int_{\M} \alpha(\sigma) F_{+}(w(\sigma)) dv_g,
\]
for every $w\in\h$. Clearly, the functional $\Phi$ and $\Psi$ are Fr\'echet differentiable.

\subsection{The $L^{\infty}$--boundness of solutions}

The next result is an application of the Nash--Moser iteration scheme and follows by a direct consequence of \cite[Theorem 3.1]{FaraciFarkas}.

\begin{proposition}\label{faracif}
  Let $(\mathcal{M},g)$ be a complete, non--compact, Riemannian
  manifold of dimension $d \geq 3$ such that conditions
  $(As_{R}^{H, g})$ and \eqref{c6nuovo} hold. Furthermore, let $V$ be
  a potential for which hypotheses $(V_1)$ and $(V_2)$ are valid. In
  addition, let
  $\alpha\in L^{\infty}(\mathcal{M})\cap L^{1}(\mathcal{M})$
  satisfying \eqref{proprietabeta} and $f\colon [0,+\infty)\to\R$ be a
  continuous function verifying \eqref{supf}--\eqref{Segno}. Then
  every critical point of the functional
  $\mathcal J_{\lambda}\colon \h\to \R$ given in \eqref{JKlambda+} is
  non--negative. Moreover, if $w_\lambda\in \h$ is a critical point of
  $\mathcal J_{\lambda}$ and $\sigma_0\in \M$, the following facts
  hold:
\begin{itemize}
\item[$(j)$] for every $\varrho>0$, $w_\lambda\in L^{\infty}(B_{\sigma_0}(\varrho))$;
\item[$(jj)$] $w_\lambda\in L^{\infty}(\M)$ and
\[
\lim_{d_g(\sigma_0,\sigma)\rightarrow +\infty} w_\lambda(\sigma)=0.
\]
\end{itemize}
\end{proposition}

\begin{proof}
Let $w_\lambda\in \h$ be a critical point of the functional \eqref{JKlambda+}.
We claim that $w_{\lambda}$ is non-negative on $\M$. Indeed, since $w_{\lambda}\in H_V^1(\M)$, then $(w_{\lambda})^-:=\max\{-w_{\lambda}, 0\}$ belongs to $H_V^1(\M)$ as well. So, taking also account of the relationship
	\begin{align*}
	\left\langle w_{\lambda},(w_{\lambda})^-\right\rangle & = \int_{\M}\left\langle \nabla_g w_{\lambda}(\sigma),\nabla_g (w_{\lambda})^-(\sigma) \right\rangle_g  dv_g\\
	&\;\;\; +  \int_{\M} \alpha(\sigma)\left\langle w_{\lambda}(\sigma), (w_{\lambda})^-(\sigma)\right\rangle_g dv_g \\
	& = - \int_{\M} \left( |\nabla_g (w_{\lambda})^- (\sigma)|^2 + \alpha(\sigma)(w_{\lambda})^-(\sigma)^2\right)  dv_g,
	\end{align*}
	 we get
	\begin{align*}
	\left\langle\mathcal J_{\lambda}'(w_\lambda), (w_{\lambda})^-\right\rangle & = \left\langle w_{\lambda},(w_{\lambda})^-\right\rangle -\lambda\int_{\mathcal M} f_+(w_{\lambda}(\sigma))(w_{\lambda})^-(\sigma) dv_g\\
	& = -\left\|(w_{\lambda})^-\right\|^2\\
	& =0.
   \end{align*}
As a result, $\|(w_{\lambda})^-\|=0$ and hence $w_{\lambda}\geq 0$ a.e. on $\mathcal M$.

Let us prove now that the function $\varphi:\M\times [0,+\infty)\rightarrow\R$ given by $\varphi(\sigma,t):=\alpha(\sigma)f_+(t)$, for every $(\sigma,t)\in \M\times \R$, satisfy, for some $k>0$ {and} $q\in (2,2^*)$, the growth condition
\begin{equation}\label{crescitaf}
|\varphi(\sigma,t)|\leq k(|t| + |t|^{q-1}),
\end{equation}
for every $(\sigma,t)\in \M\times\R$.

\indent Fix $\varepsilon>0$ and $q\in\mathopen{(}2,2^*\mathclose{)}$. In view of \eqref{supf} and \eqref{condinfinito}, there exists $\delta_\varepsilon\in \mathopen{(}0,1\mathclose{)}$ such that
\begin{equation}\label{stimaf}
|f_+(t)| \leq \varepsilon\frac{\nu_0}{\left\|\alpha\right\|_\infty}|t|,
\end{equation}
for all $0<|t|\leq \delta_\varepsilon$ and $|t|\geq \delta_\varepsilon^{-1}$.
Since the function
\[
t\mapsto \frac{|f_+(t)|}{|t|^{q-1}}
\]
is bounded on $[\delta_\varepsilon, \delta_\varepsilon^{-1}]$, for some $m_\varepsilon>0$ and for every $t\in \R$ one has
\begin{equation}\label{stimavaloreassolutof}
|f_+(t)| \leq \varepsilon\frac{\nu_0}{\left\|\alpha\right\|_\infty}|t| + m_\varepsilon |t|^{q-1}.
\end{equation}
By \eqref{stimavaloreassolutof}, bearing in mind that $\alpha\in L^{\infty}(\M)$, relation \eqref{crescitaf} immediately holds for
\[
k=\max\left\{\varepsilon\nu_0,{m_\varepsilon}{\left\|\alpha\right\|_\infty}\right\}.
\]
We can now conclude by applying~\cite[Theorem 3.1]{FaraciFarkas}.
\end{proof}

\section{Proof of Theorem \ref{generalkernel0f}}

As pointed out in the Introduction, the results of this section can be viewed as a counterpart of some recent contributions obtained by many authors in several different contexts (see, among others, the papers \cite{AMR,FFK, FK, k0,k1} and \cite{k2,k3,k4, k5}) to the case of elliptic problems defined on non--compact manifolds with asymptotically non--negative Ricci curvature. We emphasize that a key ingredient in our proof is given by Lemma \ref{immer} and Proposition \ref{JDE-F(u-0)0}. They express peculiar and intrinsic aspects of the problem under consideration.

\subsection{Non--existence for $\lambda$ small}\label{sec:non-exist}

Let us prove $(i)$ of Theorem \ref{generalkernel0f}.

Arguing by contradiction, suppose that there exists a (non--negative) weak solution $w_0\in \h\setminus\{0\}$ to problem \eqref{problema}, i.e.
\begin{equation}\label{problemaK0f45}
\int_{\mathcal{M}}\left\langle \nabla_g w_0(\sigma),\nabla_g \varphi(\sigma) \right\rangle_g dv_g +\int_{\mathcal{M}}V(\sigma)w_0(\sigma)\varphi(\sigma)dv_g=\lambda\int_{\M}\alpha(\sigma)f_+(w_0(\sigma))\varphi(\sigma)dv_g,
\end{equation}
for every $\varphi\in \h$.\par
Testing \eqref{problemaK0f45} with $\varphi:=w_0$, we have
\begin{equation}\label{au34}
\|w_0\|^2 =\lambda\int_{\M}\alpha(\sigma)f_+(w_0(\sigma))w_0(\sigma)dv_g,
\end{equation}
and it follows that

\begin{multline}\label{au345}
\int_{\M}\alpha(\sigma)f_+(w(\sigma))w_0(\sigma)dv_g  \leq  \int_{\M}\alpha(\sigma)\left|f_+(w_0(\sigma))w_0(\sigma)\right|dv_g\\
                                       \leq  \|\alpha\|_\infty\frac{c_f}{\nu_0}\int_{\M}V(\sigma)|w_0(\sigma)|^2dv_g
                                       \leq  \|\alpha\|_\infty\frac{c_f}{\nu_0}\|w_0\|^2.
\end{multline}
 \indent By using \eqref{au34} and \eqref{au345} and bearing in mind the assumption on $\lambda$ we get
\begin{equation*}
\|w_0\|^2\leq \lambda \|\alpha\|_\infty\frac{c_f}{\nu_0}\|w_0\|^2 < \|w_0\|^2,
\end{equation*}
which is a contradiction.

\begin{remark}
	The result stated in Theorem \ref{generalkernel0f} - part $(i)$ was proved in the Euclidean case in \cite{KRO} assuming that the potential $V\in L^{\infty}(\R^d)\cap L^1(\R^d)\setminus\{0\}$ is non--negative and radially symmetric.
\end{remark}

\begin{ex}\label{esempio00}
Let $(\mathcal{M},g)$ be a complete, non--compact, Riemannian manifold (with dimension $d\geq 3$)
such that conditions $(As_{R}^{H, g})$ and \eqref{c6nuovo} hold. Furthermore, let $\alpha\in L^{\infty}(\mathcal{M})\cap L^1(\M)$ satisfying \eqref{proprietabeta} and such that $\left\|\alpha\right\|_{\infty}=1$. Then, the following problem
\begin{equation*}\label{holds}
\left\{
\mkern-8mu
\begin{array}{ll}
-\Delta_gw+(1+d_g(\sigma_0,\sigma))w=\lambda \alpha(\sigma)(\arctan w)^2 & \mbox{ in } \mathcal{M}\\
w\geq 0 & \mbox{ in } \mathcal{M}
\end{array}\right.
\end{equation*}
has only the trivial solution, provided that
\[
0\leq \lambda <\left(\max_{t>0}\frac{(\arctan t)^2}{t}\right)^{-1}.
\]
\end{ex}
\subsection{Multiplicity for large $\lambda$}\label{sec:multiple}

In the present subsection we are going to apply the compact embedding results established above to
prove Theorem \ref{generalkernel0f}. The first preliminary result show the sub--quadraticity of the potential $\Psi$ defined in \eqref{psis}.

\begin{lemma}\label{subquad}
Under our assumptions on the terms $f$ and $\alpha$ stated in Theorem \ref{generalkernel0f}, one has
\begin{equation}\label{sottoquadracita}
\lim_{\left\|w\right\|\to 0}\frac{\Psi(w)}{\left\|w\right\|^2} =0\,\,\,\,\,\,\textrm{ and}\,\,\,\,\,\, \lim_{\left\|w\right\|\to \infty}\frac{\Psi(w)}{\left\|w\right\|^2}=0.
\end{equation}
\end{lemma}
\begin{proof}
\indent Inequality \eqref{stimavaloreassolutof}, in addition to \eqref{constants}, yields
\begin{align*}
|\Psi(w)| & \leq \int_{\M} \alpha(\sigma)|F_+(w(\sigma))| dv_g\\
          & \leq \int_{\M} \alpha(\sigma) \left(\frac{\varepsilon\nu_0}{2\left\|\alpha\right\|_{\infty}} |w(\sigma)|^2 + \frac{m_\varepsilon}{q} |w(\sigma)|^q\right)dv_g\\
          & \leq \int_\M \left(\frac{\varepsilon}{2}V(\sigma)|w(\sigma)|^2 + \frac{m_\varepsilon}{q}\alpha(\sigma) |w(\sigma)|^q\right)dv_g\\
					& \leq \frac{\varepsilon}{2} \int_\M V(\sigma)|w(\sigma)|^2dv_g + \frac{m_\varepsilon}{q}\left\|\alpha\right\|_{\infty} \left\|w\right\|_{q}^q\\
& \leq \frac{\varepsilon}{2} \left\|w\right\|^2 + \frac{m_\varepsilon}{q}c_q^q \left\|\alpha\right\|_{\infty} \left\|w\right\|^q,
\end{align*}
for every $w\in \h$.\par
\indent Therefore, it follows that, for every $w\in \h$,
\[
0 \leq \frac{|\Psi(w)|}{\left\|w\right\|^2} \leq \frac{\varepsilon}{2} + \frac{m_\varepsilon}{q}c_q^q\left\|\alpha\right\|_{\infty} \left\|w\right\|^{q-2}.
\]
Since $q>2$ and $\varepsilon$ is arbitrary, the first limit of \eqref{sottoquadracita} turns out to be zero.

\indent Now, if $r\in \mathopen{(}1,2\mathclose{)}$, due to the continuity of $f_+$, there also exists a number $M_\varepsilon >0$ such that
\[
\frac{|f_+(t)|}{|t|^{r-1}} \leq M_\varepsilon,
\]
for all $t\in[\delta_\varepsilon,\delta_\varepsilon^{-1}]$, where $\varepsilon$ and $\delta_\varepsilon$ are the previously introduced numbers.\par
 The above inequality, together with \eqref{stimaf}, yields
\[
|f_+(t)| \leq\varepsilon\frac{\nu_0}{\left\|\alpha\right\|_{\infty}} |t| + M_\varepsilon |t|^{r-1}
\]
for each $t\in \R$. Now $\alpha\in L^{\frac{2}{2-r}}(\M)$, indeed
\begin{equation*}
\int_{\M}|\alpha(\sigma)|^{\frac{2}{2-r}}dv_g=\int_{\M}\alpha(\sigma)^{\frac{r}{2-r}}\alpha(\sigma)dv_g\leq \|\alpha\|_{\infty}\int_{\M}\alpha(\sigma)dv_g<+\infty.
\end{equation*}
Thus, one has
\begin{align*}
|\Psi(w)| & \leq \int_{\M} \alpha(\sigma)|F_+(w(\sigma))| dv_g\\
          & \leq \int_{\M} \alpha(\sigma) \left(\frac{\varepsilon\nu_0}{2\left\|\alpha\right\|_{\infty}} |w(\sigma)|^2 + \frac{m_\varepsilon}{r} |w(\sigma)|^r\right)dv_g\\
          & \leq \int_\M \left(\frac{\varepsilon}{2}V(\sigma)|w(\sigma)|^2 + \frac{M_\varepsilon}{r}\alpha(\sigma) |w(\sigma)|^r\right)dv_g\\
					& \leq \frac{\varepsilon}{2} \int_\M V(\sigma)|w(\sigma)|^2dv_g + \frac{M_\varepsilon}{r}\left\|\alpha\right\|_{\frac{2}{2-r}} \left\|w\right\|_{r}^r\\
& \leq \frac{\varepsilon}{2} \left\|w\right\|^2 + \frac{M_\varepsilon}{r}c_r^r \left\|\alpha\right\|_{\frac{2}{2-r}} \left\|w\right\|^r,
\end{align*}
for each $w\in\h$.\par
 Therefore, it follows that
\begin{equation}\label{stimaadinfinito}
0 \leq \frac{|\Psi(w)|}{\left\|w\right\|^2} \leq \frac{\varepsilon}{2} + \frac{M_\varepsilon}{r}c_r^r \left\|\alpha\right\|_{\frac{2}{2-r}} \left\|w\right\|^{r-2},
\end{equation}
for every $w\in\h\setminus\{0\}$.\par

\noindent Since $\varepsilon$ can be chosen as small as we wish and $r\in \mathopen{(}1,2\mathclose{)}$, taking the limit for $\left\|w\right\|\to+\infty$ in \eqref{stimaadinfinito}, we have proved the second limit of \eqref{sottoquadracita}.
\end{proof}

\indent One of the main tools used along the proof of Theorem \ref{generalkernel0f} is the Mountain Pass Theorem. The compactness assumption required by this theorem is the well-known {\em Palais--Smale condition} (see, for instance, the classical book \cite{Wi}), which in our framework reads as follows:

\medskip

$\mathcal J_\lambda$ satisfies the \emph{Palais--Smale compactness condition} at level $\mu\in \RR$ (namely (PS)$_\mu$)
if any sequence $\{w_j\}_{j\in\NN}$ in $\h$ such that
$\mathcal J_\lambda(w_j)\to \mu$ and
\[
\sup\Big\{
\big|\langle\,\mathcal J_\lambda'(w_j),\varphi\,\rangle \big| :
\varphi\in \h\,,
\|\varphi\|=1\Big\}\to 0
\]
as $j\to +\infty$,
admits a subsequence strongly convergent in $\h$.

\medskip

In the case when the right--hand side in problem~\eqref{problema} satisfies the structural condition stated in Introduction, we will prove that the corresponding energy functional~$\mathcal J_\lambda$ verifies the Palais--Smale condition. More precisely,
the next energy--compactness result holds.

\begin{lemma}\label{cps}
Let $f\colon \R \to \R$ be a continuous function satisfying conditions \eqref{supf} and \eqref{condinfinito}. Then for every $\lambda>0$, the functional $\mathcal J_{\lambda}$ is bounded from below and coercive. Moreover, for every $\mu\in\R$, and every sequence $\{w_{j}\}_{j\in \N}\subset \h$ such that
\begin{equation}\label{ps1}
\mathcal J_{\lambda}(w_j)\to \mu,
\end{equation}
and
\begin{equation}\label{ps2}
\sup\Big\{\big|\langle\,\mathcal J_{\lambda}'(w_j),\varphi \rangle
\big|: \; \varphi\in \h,\,\textrm{ and}\, \|\varphi\|=1\Big\}\rightarrow 0,
\end{equation}
as $j\rightarrow +\infty$, there exists a subsequence that strongly converges in $\h$.
\end{lemma}
\begin{proof}

\noindent Fix $\lambda>0$ and $0<\varepsilon<1/\lambda $.
Due to inequality \eqref{stimaadinfinito}, it follows that
\begin{align*}
\mathcal J_{\lambda}(w) &\geq \frac{1}{2}\|w\|^2 - \lambda\int_{\M} \alpha(x)|F_+(w(\sigma))|dv_g \\
                        &\geq \frac{1}{2}\|w\|^2 - \lambda \frac{\varepsilon}{2}\left\|w\right\|^2 -\lambda\frac{M_\varepsilon}{r}c_r^r\left\|\alpha\right\|_{\frac{2}{2-r}} \left\|w\right\|^{r}\\
                        &=\frac{1}{2}\left(1-{\lambda }{\varepsilon}\right)\|w\|^2-\lambda\frac{M_\varepsilon}{r}c_r^r\left\|\alpha\right\|_{\frac{2}{2-r}} \left\|w\right\|^{r},
\end{align*}
for every $w\in \h$. Consequently, the functional $\mathcal J_{\lambda}$ is bounded from below and coercive.

Now, let us prove that the second part of our main result holds. To this end, let $\{w_{j}\}_{j\in \N}\subset \h$ be a sequence, such that conditions \eqref{ps1} and \eqref{ps2} holds and set
\[
\|\mathcal J'_{\lambda}(w_j)\|_{*}:=\sup\Big\{\big|\langle\,\mathcal J_{\lambda}'(w_j),\varphi \rangle
\big|: \; \varphi\in \h,\,\textrm{ and}\, \|\varphi\|=1\Big\}.
\]
\indent Taking into account the coercivity of $\mathcal J_{\lambda}$, the sequence $\{w_{j}\}_{j\in \N}$ is necessarily bounded in $\h$. Since $\h$ is a reflexive space, there exists a subsequence, which for simplicity we still denote $\{w_{j}\}_{j\in \N}$, such that
$w_{j}\rightharpoonup w_\infty$ in $\h$, i.e.,
\begin{equation}\label{conv0}
\int_{\M}\left(\langle \nabla_g w_j(\sigma),\nabla_g \varphi(\sigma) \rangle_g +V(\sigma)w_j(\sigma)\varphi(\sigma)\right)dv_g\to
\end{equation}
\[
\qquad\qquad\qquad\qquad\int_{\M}\left(\langle \nabla_g w_\infty(\sigma),\nabla_g \varphi(\sigma) \rangle_g +V(\sigma)w_\infty(\sigma)\varphi(\sigma)\right)dv_g,
\]
as $j\to+\infty$, for any $\varphi\in \h$.

We will prove that the sequence $\{w_{j}\}_{j\in \N}$ strongly converges to $w_\infty\in \h$. Hence, one has
\begin{equation}\label{jj}
\langle \Phi'(w_{j}),w_{j}-w_\infty\rangle = \langle \mathcal J_{\lambda}'(w_j),w_j-w_\infty\rangle + \lambda\int_{\mathcal{M}}\alpha(\sigma)f(w_j(\sigma))(w_j-w_\infty)(\sigma)dv_g,
\end{equation}
where
\begin{eqnarray*}
\langle \Phi'(w_{j}),w_{j}-w_\infty\rangle  &=& \int_{\M}\left(|\nabla_g w_j(\sigma)|^2+V(\sigma)|w_j(\sigma)|^{2}\right)\,dv_g\\
               &&\qquad\qquad- \int_{\M}\left(\langle \nabla_g w_j(\sigma),\nabla_g w_\infty(\sigma) \rangle_g +V(\sigma)w_{j}(\sigma)w_\infty(\sigma)\right)dv_g.
\end{eqnarray*}

Since $\|\mathcal J_{\lambda}'(w_j)\|_{*}\to 0$ as $j\to+\infty$, and $\{w_j-w_\infty\}_{j\in \N}$ is bounded in $\h$, owing to $|\langle\mathcal J_{\lambda}'(w_j),w_j-w_\infty\rangle|\leq\|\mathcal J_{\lambda}'(w_j)\|_{*}\|w_j-w_\infty\|$, one has
\begin{eqnarray}\label{j2}
\langle \mathcal J_{\lambda}'(w_j),w_j-w_\infty\rangle\to 0
\end{eqnarray}
as $j\to+\infty$.

Next, let us set
\[
I:=\int_{\M}\alpha(\sigma)|f_+(w_j(\sigma))||(w_j-w_\infty)(\sigma)|dv_g.
\]
By \eqref{stimavaloreassolutof} and H\"{o}lder's inequality, it follows that
\begin{align*}
 I &\leq \varepsilon \nu_0\int_{\mathcal{M}} |w_j(\sigma)||(w_j-w_\infty)(\sigma)|dv_g \\
                                                & \quad+ m_\varepsilon\left\|\alpha\right\|_{\infty} \int_{\M} |w_j(\sigma)|^{q-1}|(w_j-w_\infty)(\sigma)|dv_g \\
                                                &\leq \varepsilon\nu_0\left\|w_j\right\|_{2} \left\|w_j-w_\infty\right\|_{2}\\ & \quad+ m_\varepsilon \left\|\alpha\right\|_{\infty} \left\|w_j\right\|_{q}^{q-1}\left\|w_j-w_\infty\right\|_{q}.
\end{align*}

Since $\varepsilon$ is arbitrary and the embedding $\h\hookrightarrow L^q(\M)$  is compact thanks to Proposition \ref{faracif}, we easily have
\indent \begin{eqnarray}\label{j3}
I=\int_{\M}\alpha(\sigma)|f_+(w_j(\sigma))||(w_j-w_\infty)(\sigma)|dv_g\to 0,
\end{eqnarray}
as $j\rightarrow +\infty$.

Now, by \eqref{jj}, \eqref{j2} and \eqref{j3} we deduce that
\begin{eqnarray*}\label{fin}
\langle \Phi'(w_{j}),w_{j}-w_\infty\rangle
\rightarrow 0,
\end{eqnarray*}
as $j\rightarrow +\infty$. Hence
\begin{eqnarray}\label{fin2}
\int_{\mathcal{M}}\left(|\nabla_g w_j(\sigma)|^2+V(\sigma)|w_j(\sigma)|^2\right)dv_g-\int_{\mathcal{M}}\left\langle \nabla_g w_j(\sigma),\nabla_g w_\infty(\sigma) \right\rangle_g dv_g
\end{eqnarray}
\[
\qquad\qquad -\int_{\mathcal{M}}V(\sigma)w_j(\sigma)w_\infty(\sigma)dv_g\rightarrow 0,
\]
as $j\rightarrow +\infty$.

Thus, by \eqref{fin2} and \eqref{conv0} it follows that
\[
\lim_{j\rightarrow +\infty}\int_{\mathcal{M}}\left(|\nabla_g w_j(\sigma)|^2+V(\sigma)|w_j(\sigma)|^2\right)dv_g = \int_{\mathcal{M}}\left(|\nabla_g w_\infty(\sigma)|^2+V(\sigma)|w_\infty(\sigma)|^2\right)dv_g
\]

Thanks to \cite[Proposition III.30]{brezis}, $w_j\rightarrow w_\infty$ in $\h$. This completes the proof.
\end{proof}

The next proposition will be crucial in order to correctly precise the statements of Theorem \ref{generalkernel0f}; see also Remark \ref{parametri}.

\begin{proposition}\label{JDE-F(u-0)0} Under the hypotheses of Theorem
  \ref{generalkernel0f} there exist $t_0\in (0,+\infty)$, $0<r<\rho$
  and $\varepsilon_0\in [1/2,1)$ such that
\begin{equation}\label{u0positive}
{\Psi(t_0w_{\rho,r}^{\varepsilon_0})=\int_\M \alpha(\sigma)F_+(t_0w_{\rho,r}^{\varepsilon_0})dv_g>0,}
\end{equation}
where the function $w_{\rho,r}^{\varepsilon_0}\in \h$ is given in
\eqref{TestFunction}.
\end{proposition}

\begin{proof} By \eqref{Segno} there exists $t_0\in (0,+\infty)$ such that
$F(t_0)>0.$ Moreover, let $\varepsilon \in [1/2,1)$ and let $0<r<\rho$ be such that
\eqref{ess} holds. We have
\begin{align*}
\int_{\M}
\alpha(\sigma) F_+(t_0w_{\rho,r}^{\varepsilon}(\sigma))\,dv_g & = \int_{\M \setminus A_r^{\rho}(\sigma_0)}
\alpha(\sigma)F_+(t_0w_{\rho,r}^{\varepsilon}(\sigma))\,dv_g\\
& \quad+\int_{A_{\varepsilon r}^{\rho}(\sigma_0)}
\alpha(\sigma) F_+(t_0w_{\rho,r}^{\varepsilon}(\sigma))\,dv_g\\
&\quad+\int_{A_{r}^{\rho}(\sigma_0)\setminus A_{\varepsilon r}^{\rho}(\sigma_0)}
\alpha(\sigma)F_+(t_0w_{\rho,r}^{\varepsilon}(\sigma))\,dv_g.
\end{align*}
Hence
\begin{align*}\label{laprima}
\int_{\M}
\alpha(\sigma) F_+(t_0w_{\rho,r}^{\varepsilon}(\sigma))\,dv_g & \geq \alpha_0F(t_0)\operatorname{Vol}_g(A_{\varepsilon r}^{\rho}(\sigma_0))\\
& \quad -\int_{A_{r}^{\rho}(\sigma_0)\setminus A_{\varepsilon r}^{\rho}(\sigma_0)}
\alpha(\sigma)|F_+(t_0w_{\rho,r}^{\varepsilon}(\sigma))|\,dv_g.\nonumber
\end{align*}
Now, setting
\[
I:=\int_{A_{r}^{\rho}(\sigma_0)\setminus A_{\varepsilon r}^{\rho}(\sigma_0)}
\alpha(\sigma)|F_+(t_0w_{\rho,r}^{\varepsilon}(\sigma))|\,dv_g,
\]
by \eqref{TestFunction} one has
\[
I=\int_{A_{r}^{\rho}(\sigma_0)\setminus A_{\varepsilon r}^{\rho}(\sigma_0)}
\alpha(\sigma)\left|F_+\left(\frac{2t_0}{r} \left(r- |d_g(\sigma_0,\sigma)-\rho| \right)\right)\right|\,dv_g.
\]
Moreover, by $(j_2)$, it follows that
\begin{equation}\label{laseconda}
I\leq \|\alpha\|_{\infty}\max_{t\in(0,t_0]}
 |F(t)|\operatorname{Vol}_g(A_{r}^{\rho}(\sigma_0)\setminus A_{\varepsilon r}^{\rho}(\sigma_0)).
\end{equation}
Then, inequalities \eqref{laseconda} and \eqref{laseconda} yield
\begin{align}\label{denomi}
\int_{\M}
\alpha(\sigma) F_+(t_0w_{\rho,r}^{\varepsilon}(\sigma))\,dv_g & \geq \alpha_0F(t_0)\operatorname{Vol}_g(A_{\varepsilon r}^{\rho}(\sigma_0))\\
& \quad -\|\alpha\|_{\infty}\max_{t\in(0,t_0]}
 |F(t)|\operatorname{Vol}_g(A_{r}^{\rho}(\sigma_0)\setminus A_{\varepsilon r}^{\rho}(\sigma_0)). \nonumber
\end{align}
Since $\operatorname{Vol}_g(A_{r}^{\rho}(\sigma_0)\setminus A_{\varepsilon r}^{\rho}(\sigma_0))\rightarrow 0$, as $\varepsilon\to 1^-$, of course
\[
\|\alpha\|_{\infty}\max_{t\in(0,t_0]}
 |F(t)|\operatorname{Vol}_g(A_{r}^{\rho}(\sigma_0)\setminus A_{\varepsilon r}^{\rho}(\sigma_0))\rightarrow 0,
\]
as $\varepsilon\to 1^-$. Moreover,
\begin{equation*}
\operatorname{Vol}_g(A_{\varepsilon r}^{\rho}(\sigma_0))\rightarrow \operatorname{Vol}_g(A_{r}^{\rho}(\sigma_0))
\end{equation*}
as $\varepsilon\to 1^-$. Thus, there exists $\varepsilon_0>0$ such that
\[
\alpha_0F(t_0)\operatorname{Vol}_g(A_{\varepsilon_0 r}^{\rho}(\sigma_0))>
 \|\alpha\|_{\infty}\max_{t\in(0,t_0]}
 |F(t)|\operatorname{Vol}_g(A_{r}^{\rho}(\sigma_0)\setminus A_{\varepsilon_0 r}^{\rho}(\sigma_0)).
\]
\noindent Thus the function
$w_0^{\varepsilon_0}:=t_0w_{\rho,r}^{\varepsilon_0}\in\h$ verifies inequality (\ref{u0positive}).
\end{proof}

We are in position now to prove item $(ii)$ of Theorem \ref{generalkernel0f}.\par\smallskip
\noindent {\emph{First solution via direct minimization}} -
By Proposition \ref{JDE-F(u-0)0} the number
\begin{equation}\label{hatlambda}
\lambda^\star:= \inf_{\underset{w\in \h}{\Psi(w)>0}}\frac{\Phi(w)}{\Psi(w)}
\end{equation}
is well--defined and, owing to Lemma \ref{subquad}, one has that
$\lambda^{\star}\in (0,+\infty)$.\par
Fixing $\lambda>\lambda^\star$ and choosing $w_\lambda^\star\in \h$
with $\Psi(w_\lambda^\star)>0$ and
\[
\lambda^\star \leq \frac{\Phi(w_\lambda^\star)}{\Psi(w_\lambda^\star)}<\lambda,
\]
it follows that
\[
c_{1,\lambda}:= \inf_{w\in \h} \mathcal J_\lambda(w) \leq \mathcal J_\lambda(w_\lambda^\star)<0.
\]
\indent Since $\mathcal J_\lambda$ is bounded from below and satisfies the Palais--Smale condition (PS)$_{c_{1,\lambda}}$, it follows that  $c_{1,\lambda}$ is a critical value of $\mathcal J_\lambda$. Hence, there exists $w_{1,\lambda}\in \h\setminus\{0\}$ such that
\[
\mathcal J_\lambda(w_{1,\lambda})=c_{1,\lambda} <0 \quad \mbox{ and } \quad \mathcal J'_\lambda(w_{1,\lambda})=0.
\]
\noindent {\emph{Second solution via Mountain Pass Theorem}} - Fix $\lambda>\lambda^\star$, where the parameter $\lambda^\star$ is defined in \eqref{hatlambda}, and apply \eqref{stimavaloreassolutof} with $\varepsilon:=1/(2\lambda)$. For each $w\in \h$ one has
\begin{align*}
\mathcal J_\lambda(w) & = \frac{1}{2}\left\|w\right\|^2 -\lambda\Psi(w)\\
                   & \geq \frac{1}{2}\left\|w\right\|^2 - \frac{\lambda}{2}\varepsilon\nu_0\left\|w\right\|_2^2 - \frac{\lambda}{q}\left\|\alpha\right\|_{\infty} m_\lambda\left\|w\right\|_{q}^q\\
									 & \geq \frac{1-\lambda\varepsilon}{2} \left\|w\right\|^2 - \frac{\lambda}{q}\left\|\alpha\right\|_{\infty} m_\lambda c_q^q\left\|w\right\|^q.
\end{align*}
Setting
\[
r_\lambda:= \min\left\{\left\|w_\lambda^\star\right\|, \left(\frac{q}{4\lambda\left\|\alpha\right\|_{\infty} m_\lambda c_q^q}\right)^{1/(q-2)}\right\},
\]
since $\mathcal J_\lambda(0)=0$, it easily seen that
\[
\inf_{\left\|w\right\|=r_\lambda}\mathcal J_\lambda (w) > \mathcal J_\lambda(0) > \mathcal J_\lambda(w_\lambda^\star).
\]
Then the energy functional possesses the usual structure of the mountain pass geometry.\par
 Therefore, by using Lemma \ref{cps}, we can use the Mountain Pass Theorem to obtain the existence of $w_{2,\lambda}\in \h$ so that $\mathcal J'_\lambda (w_{2,\lambda})=0$ and $\mathcal J_\lambda (w_{2,\lambda})=c_{2,\lambda}$, where the level $c_{2,\lambda}$ has the well--known variational characterization:
\[
c_{2,\lambda}:=\inf_{\gamma\in\Gamma}\max_{t\in[0,1]}\mathcal J_\lambda (\gamma(t)),
\]
in which
\[
\Gamma:=\left\{\gamma\in C^0([0,1];\h): \gamma(0)=0, \gamma(1)=w_\lambda^\star\right\}.
\]
\indent Thanks to the fact that
\[
c_{2,\lambda}\geq \inf_{\left\|w\right\|=r_\lambda}\mathcal J_\lambda (w)>0,
\]
the existence of two distinct non--trivial weak solutions to Problem \eqref{problema} is proved.\par
 Furthermore, invoking Proposition \ref{faracif}, it follows that $w_{i,\lambda}\in L^{\infty}(\M)\setminus\{0\}$, with $i\in\{1,2\}$.
This completes the proof.

\begin{remark}\label{parametri}
A natural question arises about the parameter $\lambda^{\star}$ obtained in Theorem \ref{generalkernel0f}: can we estimate it?
Indeed, the proof of Theorem \ref{generalkernel0f} gives an exact, but quite involved form of this positive real parameter.
\noindent We give an upper estimate of it which can be easily
calculated by using the test function given in \eqref{TestFunction}. This fact can be done in terms of some analytical and geometrical constants.
Since
\[
\lambda^\star:= \inf_{\substack{w \in H^1_V(\mathcal{M}) \\ \Psi (w) > 0}}\frac{\Phi(w)}{\Psi(w)},
\]
bearing in mind that $\Psi(w_0^{\varepsilon_0})>0$, where
\begin{equation} \label{TestFunction000}
w_0^{\varepsilon_0}=t_0w_{\rho,r}^{\varepsilon}(\sigma):=\left\{
\begin{array}{ll}
0 & \mbox{ if $\sigma \in \M \setminus A_r^{\rho}(\sigma_0)$} \\[5pt]
t_0 & \mbox{ if $\sigma \in A_{\varepsilon r}^{\rho}(\sigma_0)$}\\[5pt]
 t_0\displaystyle\left(\frac{r- |d_g(\sigma_0,\sigma)-\rho|}{(1-\varepsilon)r}\right)
& \mbox{ if $\sigma \in A_{r}^{\rho}(\sigma_0)\setminus A_{\varepsilon r}^{\rho}(\sigma_0)$},
\end{array}
\right.
\end{equation}
one has
\[
\lambda^\star \leq \frac{\Phi(w_0^{\varepsilon_0})}{\Psi(w_0^{\varepsilon_0})}.
\]
Inequalities \eqref{stimazz} and \eqref{denomi} yield $\lambda^{\star}\leq \lambda_0$, where
\begin{equation}\label{lambda0}
\lambda_0:=\frac{m_{V}t_0^2\left(1+\frac{1}{(1-\varepsilon_0)^2r^2}\right)\operatorname{Vol}_g(A_r^{\rho}(\sigma_0))}{\alpha_0F(t_0)\operatorname{Vol}_g(A_{\varepsilon_0 r}^{\rho}(\sigma_0))-
 \|\alpha\|_{\infty}\max_{t\in(0,t_0]}
 |F(t)|\operatorname{Vol}_g(A_{r}^{\rho}(\sigma_0)\setminus A_{\varepsilon_0 r}^{\rho}(\sigma_0))}.
\end{equation}
Of course, if the nonlinear term $f$ is non--negative the potential
$F$ is non--decreasing. In this case
$\max_{t\in(0,t_0]} |F(t)|=F(t_0)$ and
 \begin{equation*}\label{lambda0fpositive}
\lambda_0:=\frac{m_{V}t_0^2\left(1+\frac{1}{(1-\varepsilon_0)^2r^2}\right)\operatorname{Vol}_g(A_r^{\rho}(\sigma_0))}{F(t_0)(\alpha_0\operatorname{Vol}_g(A_{\varepsilon_0 r}^{\rho}(\sigma_0))-
 \|\alpha\|_{\infty}\operatorname{Vol}_g(A_{r}^{\rho}(\sigma_0)\setminus A_{\varepsilon_0 r}^{\rho}(\sigma_0)))}.
\end{equation*}
Then, the conclusions of Theorem \ref{generalkernel0f} are valid for every $\lambda>\lambda_0$.
We also notice that
\begin{equation}\label{hatlambda2}
\lambda^\star:= \inf_{\underset{w\in\h}{\Psi(w)>0}}\frac{\Phi(w)}{\Psi(w)}\geq \frac{\nu_0}{c_f\left\|\alpha\right\|_{\infty}}.
\end{equation}
Indeed, integrating \eqref{c6}, one clearly has
\begin{equation}\label{integrata}
|F_+(t)|\leq \frac{c_f}{2}|t|^{2},\quad \forall\,t\in \R.
\end{equation}
Thus, by \eqref{integrata}, it follows that
\begin{align*}
\Psi(w) & \leq \int_{\M} \alpha(\sigma)|F_+(w(\sigma))|dv_g\\
                   & \leq c_f\frac{\|\alpha\|_{\infty}}{2\nu_0} \int_{\M}V(\sigma)|w(\sigma)|^2dv_g \\
                   & \leq c_f\frac{\|\alpha\|_{\infty}}{2\nu_0} \left(\int_{\mathcal{M}}|\nabla_g w(\sigma)|^2dv_g+\int_{\M}V(\sigma)|w(\sigma)|^2dv_g\right) \\
                   & = c_f\frac{\|\alpha\|_{\infty}}{2\nu_0} \|w\|^2,
\end{align*}
for every $w\in \h$. The above relation immediately yields
\begin{equation*}
\inf_{\underset{w\in\h}{\Psi(w)>0}}\frac{\Phi(w)}{\Psi(w)}=\frac{1}{2}\inf_{\underset{w\in\h}{\Psi(w)>0}}\frac{\|w\|^2}{\Psi(w)}\geq \frac{\nu_0}{c_f\left\|\alpha\right\|_{\infty}}.
\end{equation*}
 Hence, inequality \eqref{hatlambda2} is verified. We point out that no information is available concerning the number of solutions of problem \eqref{problema} whenever
\[
\lambda\in\left[\frac{\nu_0}{c_f\left\|\alpha\right\|_{\infty}},\lambda^{\star}\right].
\]
\end{remark}

\begin{remark}
For an explicit computation of the parameter $\lambda_0$ whose expression is given in \eqref{lambda0} it is necessary to compute the volume of certain annulus--type domains on $(\M,g)$. This fact in general presents several difficulties, since the computations are related to the volume of the geodesic balls. For instance, in the Cartan--Hadamard framework  investigated in \cite{FFK}, the authors recall some well--know facts on the asymptotic behaviour of the real function given by
\[
\phi(r):= \frac{\operatorname{Vol}_g(B_\sigma(r))}{V_{c,d}(r)},\quad\textrm{ for\,\,every\,\, } \sigma\in\M
\]
where $V_{c,d}(r)$ denotes the volume of the ball with radius $r>0$ in the
$d$--dimensional space form that is, either the hyperbolic space with sectional
curvature $c$, when $c$ is negative, or the
Euclidean space in the case $c=0$. We notice that, for every $\sigma\in \M$, one has
\[
\lim_{r\rightarrow 0^+}\phi(r)=1,
\]
and by Bishop--Gromov's result, $h$ is non--decreasing. As a byproduct of the above remarks, it follows that
\begin{equation}\label{Bishop}
\operatorname{Vol}_g(B_\sigma(r))\geq V_{c,d}(r),
\end{equation}
for every $\sigma\in\M$ and $r>0$. When equality holds in \eqref{Bishop}, the value of the sectional curvature of the manifold $(M,g)$ is $c$.
\end{remark}

\begin{remark}\label{parametri2}
Assume that condition $(As_{R}^{H, g})$ holds. Then, if $\sigma\in\M$, by \cite[Corollary 2.17]{PRS}, one has
\begin{equation*}
\frac{\operatorname{Vol}_g(B_\sigma(\rho))}{\operatorname{Vol}_g(B_\sigma(r))}\leq e^{(d-1)b_0}\left(\frac{\rho}{r}\right)^{d},
\end{equation*}
for every $0<r<\rho$, where we set
\[
b_0:=\int_0^{\infty}tH(t)dt.
\]
Moreover, if $\omega_d$ denotes the Euclidean volume of the unit ball in $\R^{d}$, the following inequality holds
\begin{equation}\label{palla1}
\operatorname{Vol}_g(B_\sigma(\rho))\leq e^{(d-1)b_0}\omega_d\rho^d,
\end{equation}
for every $\rho>0$ and $\sigma\in \M$. Then, inequality \eqref{palla1} yields
\begin{equation}\label{anello1}
\operatorname{Vol}_g(A_r^{\rho}(\sigma_0))\leq e^{(d-1)b_0}\omega_d(\rho+r)^d,
\end{equation}
and
\begin{equation}\label{anello2}
\operatorname{Vol}_g(A_r^{\rho}(\sigma_0)\setminus A_{\varepsilon_0 r}^{\rho}(\sigma_0))\leq e^{(d-1)b_0}\omega_d(\rho+r)^d.
\end{equation}
Suppose that
\begin{equation}\label{Fru3}
    \operatorname{Vol}_g(A_{\varepsilon_0 r}^{\rho}(\sigma_0))>
 \frac{e^{(d-1)b_0}\omega_d(\rho+r)^d}{\alpha_0F(t_0)}\|\alpha\|_{\infty}\max_{t\in(0,t_0]}|F(t)|.
     \end{equation}
\noindent By \eqref{anello1} and \eqref{Fru3} we notice that
\begin{equation}\label{ultima}
    \frac{\alpha_0F(t_0)}{\|\alpha\|_{\infty}\max_{t\in(0,t_0]}
 |F(t)|}>\frac{e^{(d-1)b_0}\omega_d(\rho+r)^d}{\operatorname{Vol}_g(A_{\varepsilon_0 r}^{\rho}(\sigma_0))}\geq
 \frac{\operatorname{Vol}_g(A_{r}^{\rho}(\sigma_0)\setminus A_{\varepsilon_0 r}^{\rho}(\sigma_0))}{\operatorname{Vol}_g(A_{\varepsilon_0 r}^{\rho}(\sigma_0))}.
     \end{equation}
\noindent Then \eqref{anello2} and \eqref{ultima} yield
\begin{equation}\label{anello3}
\lambda_0\leq \frac{m_{V}t_0^2\left(1+\frac{1}{(1-\varepsilon_0)^2r^2}\right)e^{(d-1)b_0}\omega_d(\rho+r)^d}{\alpha_0F(t_0)\operatorname{Vol}_g(A_{\varepsilon_0 r}^{\rho}(\sigma_0))-
 \|\alpha\|_{\infty}\max_{t\in(0,t_0]}
 |F(t)|e^{(d-1)b_0}\omega_d(\rho+r)^d}.
\end{equation}
where $\lambda_0$ is given in \eqref{lambda0}.
     Thus, by \eqref{anello3} and \eqref{ultima}, the conclusion of Theorem \ref{generalkernel0f} - part $(ii)$ is valid for every
    \[
    \lambda>\frac{m_{V}t_0^2\left(1+\frac{1}{(1-\varepsilon_0)^2r^2}\right)e^{(d-1)b_0}\omega_d(\rho+r)^d}{\alpha_0F(t_0)\operatorname{Vol}_g(A_{\varepsilon_0 r}^{\rho}(\sigma_0))-
 \|\alpha\|_{\infty}\max_{t\in(0,t_0]}
 |F(t)|e^{(d-1)b_0}\omega_d(\rho+r)^d},
    \]
    where
    $  m_{V}:=\max\left\{1,\max_{\sigma\in A_r^{\rho}(\sigma_0)}V(\sigma)\right\}$,
    $0<r<\rho$ such that
    \[
    \operatorname{ess\, inf}\limits_{\sigma\in A_{r}^{\rho}}\alpha(\sigma)\geq \alpha_0>0,
    \]
    $t_0\in (0,+\infty)$ and $F(t_0)>0$ with
    \begin{equation*}
    \frac{\alpha_0F(t_0)}{\|\alpha\|_{\infty}\max_{t\in(0,t_0]}
 |F(t)|}>\frac{e^{(d-1)b_0}\omega_d(\rho+r)^d}{\operatorname{Vol}_g(A_{\varepsilon_0 r}^{\rho}(\sigma_0))},
     \end{equation*}
     for some $\varepsilon_0\in [1/2,1)$.
\end{remark}

\begin{remark}\label{parametri3}
A different form of Theorem \ref{generalkernel0f} can be obtained in the Euclidean case (see Theorem \ref{boccia}). In such a case $(As_{R}^{H, g})$ holds with $H=0$.
A concrete upper bound (namely $\lambda^{0}_E$) for the parameter $\lambda^{\star}_E$ can be obtained exploiting \eqref{lambda0} and observing that, in this context, if $a<b$, one has
\begin{equation*}
\operatorname{Vol}_g(A_a^{b}(\sigma_0))=\omega_d((a+b)^d-(b-a)^d),
\end{equation*}
where $\omega_d$ is the measure of the unit ball in $\R^d$ and $\sigma_0$ is the origin. More precisely, a direct computation ensures that
\begin{equation}\label{lambda0E}
\lambda^{0}_E:=\frac{m_{V}t_0^2\left(1+\frac{1}{(1-\varepsilon_0)^2r^2}\right)((r+\rho)^d-(\rho-r)^d)}{\lambda_D(t_0,\varepsilon_0,r,\rho,\alpha,f,d)},
\end{equation}
where the expression of $\lambda_D(t_0,\varepsilon_0,r,\rho,\alpha,f,d)$ has the following form
\begin{multline*}
\lambda_D(t_0,\varepsilon_0,r,\rho,\alpha,f,d):=
\alpha_0F(t_0)((\varepsilon_0r+\rho)^d-(\rho-\varepsilon_0r)^d) \\
 {}-\|\alpha\|_{\infty}\max_{t\in(0,t_0]}
 |F(t)|[((r+\rho)^d-(\rho-r)^d))-((\varepsilon_0r+\rho)^d-(\rho-\varepsilon_0r)^d)].
\end{multline*}
\end{remark}
A simple consequence of Theorem \ref{boccia} is the following application in the Euclidean setting.
\begin{ex}\label{esempio1}
Let $\R^d$ be the $d$--dimensional $(d\geq 3)$ Euclidean space and let $V\in C^0(\R^d)$ be a potential such that $v_0=\inf_{x\in \R^d}V(x)>0$, and
\[
\lim_{|x|\rightarrow +\infty}\int_{B(x)}\frac{dy}{V(y)}=0,
\]
where $B(x)$ denotes the unit ball of center $x\in \R^{d}$. Then, owing to Corollary \ref{boccia}, there exists $\lambda^{0}_E>0$ such that the following problem
\begin{equation*}
\left\{
\begin{array}{ll}
-\Delta u+V(x)u=\displaystyle\frac{\lambda}{(1+|x|^d)^2} \log(1+u^2) & \mbox{ in } \R^d\\
u\geq 0 & \mbox{ in } \R^d\\
u\rightarrow 0 & \mbox{ as } |x|\rightarrow +\infty,
\end{array}\right.
\end{equation*}
admits at least two distinct and non--trivial weak solutions $u_{1,\lambda}, u_{2,\lambda}\in L^{\infty}(\R^d) \cap H_V^{1}(\R^d)$, provided that $\lambda>\lambda^0_E$. A concrete value of $\lambda^0_E$ can be computed by using \eqref{lambda0E}.
\end{ex}

We end this section giving a short and not detailed proof of Theorem \ref{generalkernel0f222}.\par

\medskip

\noindent \emph{Proof of Theorem \ref{generalkernel0f222}} - Fix $\lambda>0$. On account of Lemma \ref{immer} and hypotheses \eqref{supf}, as well as \eqref{supg} and \eqref{eq:mu0}, one can apply in a standard manner the Mountain Pass Theorem of Ambrosetti--Rabinowitz obtaining a critical point $w\in \h$ with positive energy--level. For this, we have to check that the functional $\mathcal J_\lambda$, given in \eqref{JKlambda+},
 has a particular geometric structure (as stated, e.g., in conditions~($1^\circ$)--($3^\circ$) of \cite[Theorem~6.1]{struwe})
and satisfies the Palais--Smale compactness condition (see, for instance, \cite[page~70]{struwe}). Furthermore, by Proposition \ref{faracif}, it follows that $w\in L^{\infty}(\M)\setminus\{0\}$. We notice that the analysis that we perform here in order to use
Mountain Pass Theorem is quite general and may be suitable for other goals
too. See, for instance, \cite[Theorem 1.1 - part $(i)$]{BK}.

\section{Proof of Theorem \ref{Main1}}

A non--local counterpart of Theorem \ref{Main1} was proved in \cite{GMBSe} and the literature is rich of many
extensions to elliptic problems defined on bounded Euclidean domains. Moreover, a suitable Euclidean version of Theorem \ref{Main1}
was recently performed in \cite{Molica} for Schr\"{o}dinger equations in presence of symmetries. In the sequel, we
will adapt the methods developed in \cite{Molica, GMBSe} to our non--compact abstract manifolds setting. Of course, when using direct minimization, we need that the functional
is bounded from below and, in this case, we look for its global minima,
which are the most natural critical points. In looking for global minima of a
functional, the two relevant notions are the weakly lower semicontinuity and
the coercivity, as stated in the Weierstrass Theorem.
The coercivity of the functional assures that the minimizing sequence is
bounded, while the semicontinuity gives the existence of the minimum for
the functional. We would emphasize the fact that the result presented here is not be achieved by direct minimization arguments.
See \cite{MoRa}, for related topics.
\subsection{Some preliminary Lemmas}

\begin{proposition}\label{C1}
Assume that $f:[0,+\infty)\rightarrow \mathbb{R}$ is a continuous function such that condition \eqref{crescita2}
holds, for some $q\in \left[2,2^*\right]$. Let $\Psi:\h\rightarrow \mathbb{R}$ defined by
\[
\Psi(w):=\int_{\M}\alpha(\sigma)F_+(w(\sigma))dv_g,\quad \forall\, w\in \h
\]
with $\alpha\in L^{\infty}(\M)_{+}\cap L^{p}(\M)$, where $p:=q/(q-1)$ is the conjugate Sobolev exponent of $q$, and $F_+(t):=\int_0^{t}f_+(\tau)d\tau$, for every $t\in \R$, with $f_+$ given in \eqref{fpiu}. Then, the functional $\Psi$ is continuously G\^{a}teaux differentiable on $\h$.
\end{proposition}

\begin{proof}
 First of all we notice that the functional $\Psi$ is well--defined. Indeed, owing to \eqref{crescita2}, bearing in mind the assumptions on the weight $\alpha$ and by using the H\"{o}lder inequality one has that
\begin{align}\label{ineq}
  \int_{\M} \alpha(\sigma)F_+(w(\sigma))dv_g &\leq  \beta_f\int_{\M}\alpha(\sigma)|w(\sigma)|dv_g+\frac{\beta_f}{q}\int_{\M}\alpha(\sigma)|w(\sigma)|^{q}dv_g\nonumber\\
   &\leq \beta_f\left(\int_{\M}|\alpha(\sigma)|^pdx\right)^{1/p}\left(\int_{\M}|w(\sigma)|^qdv_g\right)^{1/q}\\
   &\qquad{}+\frac{\beta_f}{q}\|\alpha\|_{\infty}\int_{\M}|w(\sigma)|^qdv_g,\nonumber
\end{align}
for every $w\in \h$.\par
Since the Sobolev space $\h$ is continuously embedded in the Lebesgue space $L^q(\M)$, inequality \eqref{ineq} yields
\begin{equation*}\label{ineq2}
\Psi(w) \leq \beta_f\left(\|\alpha\|_p\|w\|_q+\frac{1}{q}\|\alpha\|_{\infty}\|w\|_q^q\right)<+\infty,
\end{equation*}
for every $w\in \h$.\par
Now, let us compute the G\^{a}teaux derivative $\Psi':\h\rightarrow (\h)^{*}$ such that $w\mapsto \langle \Psi'(w),v\rangle$, where
\begin{align*}
\langle \Psi'(w),v\rangle &:=\lim_{h\rightarrow 0}\frac{\Psi(w+hv)-\Psi(w)}{h} \\
&=\lim_{h\rightarrow 0}\frac{\displaystyle\int_{\M}\alpha(\sigma)(F_+(w(\sigma)+hv(\sigma))-F_+(w(\sigma)))dv_g}{h}.
\end{align*}
We claim that
\begin{equation}\label{plimite}
\lim_{h\rightarrow 0^+}\frac{\displaystyle\int_{\M}\alpha(\sigma)(F_+(w(\sigma)+hv(\sigma))-F_+(w(\sigma)))dv_g}{h}=
\int_{\M}\alpha(\sigma)f_+(w(\sigma))v(\sigma)dv_g.
\end{equation}
\indent Indeed, for a.e. $\sigma\in \M$, and $|h|\in (0,1)$, by the mean value theorem, there exists $\theta\in (0,1)$ (depending of $\sigma$) such that (up to the constant $\beta_f$)
\begin{align}\label{ineq3}
  \alpha(\sigma)|f(w(\sigma)+\theta hv(\sigma))||v(\sigma)| &\leq  \alpha(\sigma)(1+|w(\sigma)+\theta h v(\sigma)|^{q-1})|v(\sigma)|\nonumber\\
   &\leq \alpha(\sigma)|v(\sigma)|+\|\alpha\|_\infty||w(\sigma)|+|v(\sigma)||^{q-1}|v(\sigma)|\nonumber\\
   &\leq\alpha(\sigma)|v(\sigma)|+2^{q-2}\|\alpha\|_\infty(|w(\sigma)|^{q-1}+|v(\sigma)|^{q-1})|v(\sigma)|.\nonumber
\end{align}
Let us set
\[
g(\sigma):=\alpha(\sigma)|v(\sigma)|+2^{q-2}\|\alpha\|_\infty(|w(\sigma)|^{q-1}+|v(\sigma)|^{q-1})|v(\sigma)|,
\]
for a.e. $\sigma\in \M$. Thanks to $q\in \left[2,2^*\right]$,
the H\"{o}lder inequality yields
\begin{eqnarray}
  \int_{\M} g(\sigma)dv_g &\leq & \|\alpha\|_p\|v\|_q+2^{q-2}\|\alpha\|_{\infty}\|w\|_{q}^{q-1}\|v\|_q\nonumber\\
   &&+2^{q-2}\|\alpha\|_{\infty}\|v\|_q^q<+\infty.\nonumber
\end{eqnarray}
\indent Thus $g\in L^1(\M)$ and the Lebesgue's Dominated Convergence Theorem ensures that \eqref{plimite} holds true.\par
We prove now the continuity of the G\^{a}teaux derivative $\Psi'$. Assume that $w_{j}\rightarrow w$ in $\h$ as $j\rightarrow +\infty$ and let us prove that
\begin{equation*}\label{plimite234finaletesi}
\lim_{j\rightarrow +\infty}\|\Psi'(w_j)-\Psi'(w)\|_{(\h)^*}=0,
\end{equation*}
where
\[
\|\Psi'(w_j)-\Psi'(w)\|_{(\h)^*} := \sup_{\|v\|=1}|\langle \Psi'(w_j)-\Psi'(w),v\rangle|.
\]
\indent Now, observe that
\begin{align}\label{norms}
  \|\Psi'(w_{j})-\Psi'(w)\|_{(\h)^*} &:= \sup_{\|v\|=1}|\langle \Psi'(w_{j})-\Psi'(w),v\rangle|\nonumber\\
   &\leq \sup_{\|v\|=1}\left|\int_{\M}\alpha(\sigma)(f_+(u_{j}(\sigma))-f_+(u(\sigma)))v(\sigma)dv_g\right|\\
   &\leq \sup_{\|v\|=1}\int_{\M}\alpha(\sigma)|(f_+(w_{j}(\sigma))-f_+(w(\sigma)))||v(\sigma)|dv_g.\nonumber
\end{align}

The Sobolev imbedding result ensures that $w_{j}\rightarrow w$ in $L^{q}(\M)$, for every $q\in [2,2^*]$, and $w_{j}(\sigma)\rightarrow w(\sigma)$ a.e. in $\M$ as $j\rightarrow +\infty$. Then, up to a subsequence, $w_{j}(\sigma)\rightarrow w(\sigma)$ a.e. in $\M$ as $j\rightarrow +\infty$, and there exists $\eta\in L^{q}(\M)$ such that $|w_{j}(\sigma)|\leq \eta(\sigma)$ a.e. in $\M$ and $|w(\sigma)|\leq \eta(\sigma)$ a.e. in $\M$ (see \cite{KRV} for more details). In particular,
\begin{align*}
&f_+(w_j(\sigma))\rightarrow f_+(w(\sigma)), \ \ \textrm{ a.e.} \ \ \textrm{ in} \ \ \M\\
&|f_+(w_j(\sigma))|\leq \beta_f(1+\eta(\sigma)^{q-1}), \ \ \textrm{ a.e.} \ \ \textrm{ in} \ \ \M.
\end{align*}

Note that the function $\M \ni x\mapsto 1+\eta(\sigma)^{q-1}$ is $p:=\frac{q}{q-1}$-summable
on every bounded measurable subset of $\M$. Therefore, by the Lebesgue Dominated Convergence
Theorem, we infer
\begin{align*}
f_+(w_j)\rightarrow f_+(w), \ \ \textrm{ in}  \ \ L^p(B_{\sigma_0}(\varrho)),
\end{align*}
for every $\varrho>0$, where $B_{\sigma_0}(\varrho)$ denotes the open geodesic ball of radius $\varrho$ centered at $\sigma_0$.\par
\indent At this point, set $\ell_q:=\sup_{\|v\|\leq 1}\|v\|_q$ and let $\varrho_\varepsilon>0$ be such that
\begin{align}\label{formula1}
\int_{\M\setminus B_{\sigma_0}({\varrho_\varepsilon})}|\alpha(\sigma)|^p
dv_g<\left(\frac{\varepsilon}{8\beta_f\ell_q}\right)^p,
\end{align}

\noindent as well as

\begin{align}\label{formula2}
\int_{\M\setminus
B_{\sigma_0}({\varrho_\varepsilon})}|\eta(\sigma)|^p
dv_g<\left(\frac{\varepsilon}{8\beta_f\ell_q\|\alpha\|_\infty}\right)^p.
\end{align}

Moreover, let $j_\varepsilon \in\N$ be such that
\begin{eqnarray}\label{formula3}
\|f_+(w_j)- f_+(w)\|_{L^p(B_{\sigma_0}({\varrho_\varepsilon}))}<\frac{\varepsilon}{4\ell_q\|\alpha\|_\infty},
\end{eqnarray}
for each $j\in \N,$ with $j\geq j_\varepsilon$.\par

Owing to \eqref{formula1}, \eqref{formula2} and \eqref{formula3}, for $j\geq j_\varepsilon$, inequality \eqref{norms} yields

\begin{align}\label{norms22}
  \|\Psi'(w_{j})-\Psi'(w)\|_{(\h)^*} &\leq \sup_{\|v\|\leq 1}\int_{\M\setminus
B_{\sigma_0}({\varrho_\varepsilon})}\alpha(\sigma)|f_+(w_j(\sigma))-
f_+(w(\sigma))||v(\sigma)|dv_g\nonumber\\
   &\leq  2\beta_f \sup_{\|v\|\leq 1}\int_{\M\setminus
B_{\sigma_0}({\varrho_\varepsilon})}\alpha(\sigma)(1+|\eta(\sigma)|^{q-1})|v(\sigma)|dv_g\nonumber\\
   & {}+\|\alpha\|_\infty \sup_{\|v\|\leq
1}\int_{B_{\sigma_0}({\varrho_\varepsilon})}|f_+(w_j(\sigma))- f_+(w(\sigma))||v(\sigma)|dv_g\nonumber\\
&\leq 2\beta_f\ell_q\|\alpha\|_{L^p(\M\setminus B_{\sigma_0}({\varrho_\varepsilon})}+2\beta_f\ell_q\|\alpha\|_\infty
\|\eta\|_{L^p(\M\setminus B_{\sigma_0}({\varrho_\varepsilon}))}\nonumber\\
&{}+\ell_q\|\alpha\|_\infty\|f_+(w_j)-
f_+(w)\|_{L^p(B_{\sigma_0}({\varrho_\varepsilon})}<\varepsilon,
\end{align}
that concludes the proof.
\end{proof}

\begin{lemma}\label{semicontinuity}

Assume that $f\colon \mathbb{R}\rightarrow \mathbb{R}$ is a continuous function such that
condition \eqref{crescita2} holds for every $q\in \left(2,2^*\right)$. Then, for every $\lambda>0,$ the functional $\mathcal
J_\lambda$ is sequentially weakly lower semicontinuous on $\h$.
\end{lemma}
\emph{Proof.}  First, on account of Brezis \cite[Corollaire
III.8]{brezis}, the functional
$\Phi$ is sequentially weakly lower semicontinuous on $\h$. In order to prove that $\Psi$ is weakly continuous, we assume that there exists a sequence $\{w_j\}_{j\in \mathbb N}\subset \h$ which weakly converges to an element $w_0\in \h$, and such that
\begin{equation}\label{wslsc}
|\Psi(w_j)-\Psi(w_0)|>\delta,
\end{equation}
for every $j\in \mathbb{N}$ and some $\delta>0$.\par
 Since $\{w_j\}_{j\in \mathbb{N}}$ is bounded in $\h$ and, taking into account that, thanks to Lemma \ref{immer}, $w_j\rightarrow w_0$ in $L^q(\M)$, the mean value theorem, the growth condition \eqref{crescita2} and the H\"{o}lder inequality, yield
\begin{align}\label{norms24444}
  |\Psi(w_{j})-\Psi(w_0)| &\leq \int_{\M}\alpha(\sigma)\left|F_+(w_j(\sigma))-F_+(w_0(\sigma))\right|dv_g\nonumber\\
   &\leq \beta_f\int_{\M}\alpha(\sigma)\left(2+|w_j(\sigma)|^{q-1}+|w_0(\sigma)|^{q-1}\right)|w_j(\sigma)-w_0(\sigma)|dv_g \nonumber\\
   &\leq \beta_f(2\|\alpha\|_p\|w_j-w_0\|_q+\|\alpha\|_{\infty}(\|w_j\|_{q}^{q-1}+\|w_0\|_{q}^{q-1})\|w_j-w_0\|_q)\nonumber\\
   &\leq \beta_f(2\|\alpha\|_2+\|\alpha\|_{\infty}(M+\|w_0\|_{q}^{q-1}))\|w_j-w_0\|_q,
\end{align}
for some $M>0$.\par
Since the last expression in \eqref{norms24444} tends to zero, this fact
contradicts (\ref{wslsc}). In conclusion, the functional $\Psi$ is sequentially
weakly continuous and this completes the proof.\par
\subsection{A local minimum geometry}

In order to find critical points for $\mathcal{J}_\lambda$ we will apply the following critical point theorem proved in \cite[Theorem 2.1]{VPricceri} and recalled here in a more convenient form (see \cite{GMBSe} for details).

\begin{theorem}\label{BMB}
Let $X$ be a reflexive real Banach space, and let $\Phi,\Psi:X\to\R$
be two G\^{a}teaux differentiable functionals such that $\Phi$ is
strongly continuous, sequentially weakly lower semicontinuous and coercive. Further, assume that $\Psi$
is sequentially weakly upper semicontinuous. For every $r>\inf_X
\Phi$, put
\[
\varphi(r):=\inf_{u\in\Phi^{-1}((-\infty,r))}\frac{\sup_{v\in\Phi^{-1}((-\infty,r))}\Psi(v)-\Psi(u)}{r-\Phi(u)}.
\]
Then, for each $r>\inf_X\Phi$ and each
$\lambda\in\left(0,{1}/{\varphi(r)}\right)$, the restriction
of $\mathcal{J}_\lambda:=\Phi-\lambda\Psi$ to
$\Phi^{-1}((-\infty,r))$ admits a global minimum, which is a
critical point $($local minimum$)$ of $\mathcal{J}_\lambda$ in $X$.
\end{theorem}

For our purposes, we consider the functional~$\mathcal
J_{\lambda}\colon \h\rightarrow \R$ defined by:
\[
\mathcal J_{\lambda}(w)= \Phi(w)-\lambda \Psi(w),
\]
where
\[
\Phi(w):=\frac 1 2 \left(\int_{\M}|\nabla_g w(\sigma)|^2dv_g+\int_{\M}V(\sigma)|w(\sigma)|^2dv_g\right),
\]
and
\[
\Psi(w):= \int_{\M}\alpha(\sigma)F(w(\sigma))dv_g.
\]
\indent Let $\chi\colon (0,+\infty)\rightarrow [0,+\infty)$ be the real function given by
\[
\chi(r):=\frac{\sup_{w\in\Phi^{-1}((-\infty,r))}\Psi(w)}{r},
\]
for every $r>0$. The following result gives an estimate of the function $\chi$.
\begin{proposition}For every $r>0$, one has
\begin{equation}\label{n}
\chi(r)\leq c_q\beta_f\left(\|\alpha\|_p\sqrt{\frac{2}{r}}+\frac{2^{q/2}c_q^{q-1}}{q}\|\alpha\|_{\infty}r^{q/2-1}\right).
\end{equation}
\end{proposition}
\begin{proof}
Taking into account the growth condition expressed by \eqref{crescita2} it follows that
\[
\Psi(w)=\int_{\M}\alpha(\sigma)F(w(\sigma))dv_g\leq \beta_f\int_{\M}\alpha(\sigma)|w(\sigma)|dv_g+\frac{\beta_f}{q}\int_{\M}\alpha(\sigma)|w(\sigma)|^qdv_g.
\]
\indent Moreover, one has
\begin{equation}\label{gio}
\|w\|< \sqrt{2r},
\end{equation}
\noindent for every $u\in\h$ and $\Phi(w)<r$.\par
\indent Now, by using (\ref{gio}), inequality \eqref{constants} yields
\[
\Psi(w)< c_q\beta_f\left(\|\alpha\|_p\sqrt{{2r}}+\frac{c_q^{q-1}}{q}\|\alpha\|_{\infty}(2r)^{q/2}\right),
\]
for every $w\in \h$ such that $\Phi(w)<r$.\par
\indent Hence
\[
\sup_{w\in\Phi^{-1}((-\infty,r))}\Psi(w)\leq c_q\beta_f\left(\|\alpha\|_p\sqrt{{2r}}+\frac{c_q^{q-1}}{q}\|\alpha\|_{\infty}(2r)^{q/2}\right).
\]
\noindent Then, the above inequality immediately gives \eqref{n}.
\end{proof}

Let $h\colon [0,+\infty)\rightarrow [0,+\infty)$ be the real function defined by
\[
h(\gamma):=\frac{q}{\beta_fc_q}\left(\frac{{\gamma}}{ q\sqrt{2}\|\alpha\|_p+2^{q/2}c_q^{q-1}\|\alpha\|_{\infty}\gamma^{q-1}}\right).
\]
\noindent Since $h(0)=0$ and $\lim_{\gamma\rightarrow +\infty}h(\gamma)=0$ (note that $q\in (2,2^*)$) we set
\[
\lambda_\star:=\max_{\gamma\geq 0}h(\gamma),
\]
whose explicit expression is given in relation \eqref{lambda1}.

 With the above notations the following estimate holds.
\begin{lemma}\label{intermedio}
For every $\lambda\in (0,\lambda_{\star})$, there exists $\bar{\gamma}>0$ such that
\begin{equation}\label{nov}
\chi(\bar{\gamma}^2)<\frac{1}{\lambda}.
\end{equation}
\end{lemma}

\begin{proof}
\indent Since $0<\lambda<\lambda_{\star}$, bearing in mind \eqref{lambda1},
there exists $\bar{\gamma}>0$ such that
\begin{equation}\label{n3}
\lambda<\lambda_{\star}{(\bar{\gamma})}:=\frac{q}{c_q\beta_f}\left(\frac{{\bar\gamma}}{ q\sqrt{2}\|\alpha\|_p+2^{q/2}c_q^{q-1}\|\alpha\|_{\infty}\bar\gamma^{q-1}}\right).
\end{equation}
\indent On the other hand, evaluating inequality \eqref{n} in $r=\bar\gamma^2$, we also have
\begin{equation}\label{n2}
\chi(\bar{\gamma}^2)\leq c_q\beta_f\left(\sqrt{{2}}\frac{\|\alpha\|_p}{\bar\gamma}+\frac{2^{q/2}c_q^{q-1}}{q}\|\alpha\|_{\infty}{\bar\gamma}^{q-2}\right).
\end{equation}
Inequalities \eqref{n3} and \eqref{n2} yield \eqref{nov}.
\end{proof}

\emph{Proof of Theorem \ref{Main1}} - The main idea of the proof consists in applying Theorem \ref{BMB}, to the functional~$\mathcal J_{\lambda}$. First of all, note that $\h$ is a Hilbert space and the functionals~$\Phi$ and
$\Psi$ have the regularity required by Theorem \ref{BMB} (see Lemma~\ref{semicontinuity}). Moreover, it is clear that the functional $\Phi$ is strongly continuous, coercive in $\h$ and
\[
\inf_{w\in \h}\Phi(w)=0.
\]
Since $\lambda\in (0,\lambda_\star)$, by Lemma \ref{intermedio} , there exists $\bar \gamma>0$ such that inequality \eqref{nov} holds. Now, it is easy to note that
\[
\varphi(\bar{\gamma}^2):=\inf_{w\in\Phi^{-1}((-\infty,\bar{\gamma}^2))}\frac{\sup_{v\in\Phi^{-1}((-\infty,\bar{\gamma}^2))}\Psi(v)-\Psi(w)}{r-\Phi(w)}\leq \chi(\bar{\gamma}^2),
\]
owing to $w_0\in\Phi^{-1}((-\infty,\bar{\gamma}^2))$ and $\Phi(w_0)=\Psi(w_0)=0$, where $w_0\in\h$ is the zero function.\par
 By \eqref{nov}, it follows that
\begin{equation}\label{n2fg}
\varphi(\bar{\gamma}^2)\leq \chi(\bar{\gamma}^2)\leq \beta_fc_q\left(\sqrt{{2}}\frac{\|\alpha\|_p}{\bar\gamma}+\frac{2^{q/2}c_q^{q-1}}{q}\|\alpha\|_{\infty}{\bar\gamma}^{q-2}\right)<\frac{1}{\lambda}.
\end{equation}
Inequalities in \eqref{n2fg} give
\[
\lambda\in \left(0,\frac{q}{\beta_fc_q}\left(\frac{{\bar\gamma}}{ q\sqrt{2}\|\alpha\|_p+2^{q/2}c_q^{q-1}\|\alpha\|_{\infty}\bar\gamma^{q-1}}\right)\right)\subseteq \left(0,\frac{1}{\varphi(\bar{\gamma}^2)} \right).
\]

\indent The variational principle in Theorem \ref{BMB} ensures the existence of $w_\lambda\in\Phi^{-1}((-\infty,\bar{\gamma}^2))$ such that
\[
\Phi'(w_\lambda)-\lambda \Psi'(w_\lambda)=0.
\]
 Moreover, $w_\lambda$ is a global minimum of the restriction of the functional $\mathcal{J}_{\lambda}$ to the sublevel $\Phi^{-1}((-\infty,\bar{\gamma}^2))$.

Now, we have to show that the
solution~$w_\lambda$ found here above is not the trivial (identically zero) function. Let us fix
$\lambda\in (0, \lambda_{\star}{(\bar{\gamma})})$ for some $\bar \gamma>0$. Finally,
let $w_\lambda$ be such that
\begin{equation}\label{minimum}
\mathcal J_{\lambda}(w_{\lambda})\leq \mathcal
J_{\lambda}(w),\quad \mbox{for any}\,\,\, w\in \Phi^{-1}((-\infty,\bar\gamma^{2})),
\end{equation}
and
\begin{equation*}\label{minimum2}
w_\lambda\in\Phi^{-1}((-\infty,\bar\gamma^{2}))\,,
\end{equation*}
and also $w_\lambda$ is  a critical point of $\mathcal
J_{\lambda}$ in $\h$.\par

  In order to prove that $w_\lambda\not \equiv 0$ in $\h$\,,
first we claim that there exists a sequence of functions $\big\{w_j\big\}_{j\in
\NN}$ in $\h$ such that
\begin{equation}\label{wjBlu}
\limsup_{j\to
+\infty}\frac{\Psi(w_j)}{\Phi(w_j)}=+\infty\,.
\end{equation}

By the assumption on the limsup in~(\ref{azero}) there exists a sequence
$\{t_j\}_{j\in\NN}\subset (0,+\infty)$ such that $t_j\to 0^+$ as $j\to
+\infty$ and
\begin{equation*}\label{limsup}
 \lim_{j\rightarrow +\infty}\frac{F(t_j)}{t_j^2}=+\infty,
\end{equation*}
namely, we have that for any $M>0$ and $j$
sufficiently large
\begin{equation}\label{M}
F(t_j)\geq Mt_j^2\,.
\end{equation}

\indent Now, define $w_j:=t_jw_{\rho,r}^{\varepsilon}$ for any $j\in \NN$, where the function $w_{\rho,r}^{\varepsilon}$ is given in \eqref{TestFunction}. Of course one has that $w_j\in \h$ for any $j\in \NN$. Furthermore, taking into account the algebraic properties of the functions $w_{\rho,r}^{\varepsilon}$ stated in $i_1)$--$i_3)$, since $F_+(0)=0$, and by using (\ref{M}) we can write:
\begin{multline}\label{conto1}
 {\frac{\Psi(w_j)}{\Phi(w_j )}}= \frac{ \displaystyle\int_{A_{\varepsilon r}^{\rho}(\sigma_0)}\alpha(\sigma)F_+(w_j(\sigma))\, dv_g+\displaystyle\int_{A_{r}^{\rho}(\sigma_0)\setminus A_{\varepsilon r}^{\rho}(\sigma_0)} \alpha(\sigma)F_+(w_j(\sigma))\,dv_g}{ \Phi(w_j)}  \\
 = \frac{ \displaystyle\int_{A_{\varepsilon r}^{\rho}(\sigma_0)}\alpha(\sigma)F(t_j)\, dv_g+\displaystyle\int_{A_{r}^{\rho}(\sigma_0)\setminus A_{\varepsilon r}^{\rho}(\sigma_0)}
\alpha(\sigma)F(t_jw_{\rho,r}^{\varepsilon}(\sigma))\,dv_g}{\Phi(w_j)}  \\
\geq  2\frac{\alpha_0 M t^2_j\operatorname{Vol}_g(A_{\varepsilon r}^{\rho}(\sigma_0))+{\displaystyle\int_{A_{r}^{\rho}(\sigma_0)\setminus A_{\varepsilon r}^{\rho}(\sigma_0)}
\alpha(\sigma)F(t_jw_{\rho,r}^{\varepsilon}(\sigma))\,dv_g}}{t_j^2\|w_{\rho,r}^{\varepsilon}\|^2},
\end{multline}
for $j$ sufficiently large.

Now we have to distinguish two different cases, i.e. the case when the
lower limit in (\ref{azero}) is infinite and the one in which the
lower limit in (\ref{azero}) is finite.

\textbf{Case 1:} Suppose that
$\lim_{t\to 0^+}\frac{F(t)}{t^2}=+\infty$.

Then, there exists $\rho_M>0$ such that for any $t$ with $0<t<\rho_M$
\begin{equation}\label{H}
F(t) \geq Mt^2.
\end{equation}

Since $t_j\to 0^+$ and $0\leq w_{\rho,r}^{\varepsilon}(\sigma)\leq 1$ in $\M$, then $w_j(\sigma)=t_jw_{\rho,r}^{\varepsilon}(\sigma)\to 0^+$ as $j\to +\infty$ uniformly in $\sigma\in \M$. Hence, $0\leq w_j(\sigma)< \rho_M$ for $j$ sufficiently large and for any $\sigma\in \M$. Hence, as a consequence of (\ref{conto1}) and (\ref{H}), we have
that
\begin{multline*}
{\frac{\Psi(w_j)}{\Phi(w_j )}  \geq 2\frac{\alpha_0 M
t_j^2\operatorname{Vol}_g(A_{\varepsilon r}^{\rho}(\sigma_0))+{\displaystyle\int_{A_{r}^{\rho}(\sigma_0)\setminus A_{\varepsilon r}^{\rho}(\sigma_0)} \alpha(\sigma)F(t_jw_{\rho,r}^{\varepsilon}(\sigma))\,dx}}{t_j^2\|w_{\rho,r}^{\varepsilon}\|^2}} \\
{ \geq 2\alpha_0M\frac{\operatorname{Vol}_g(A_{\varepsilon r}^{\rho}(\sigma_0))+{\displaystyle\int_{A_{r}^{\rho}(\sigma_0)\setminus A_{\varepsilon r}^{\rho}(\sigma_0)}
|w_{\rho,r}^{\varepsilon}(\sigma)|^2\,dv_g}}{\|w_{\rho,r}^{\varepsilon}\|^2}},
\end{multline*}
for $j$ sufficiently large. The arbitrariness of $M$ gives
(\ref{wjBlu}) and so the claim is proved.

\medskip

\textbf{Case 2:} Suppose that $\liminf_{t\to 0^+}\frac{F(t)}{t^2}=\ell\in
\RR$\,.

Then, for any $\epsilon>0$ there exists $\rho_\epsilon>0$ such that for any $t$ with $0<t<\rho_\epsilon$
\begin{equation}\label{epsilon}
F(t) \geq(\ell-\epsilon)t^2\,.
\end{equation}
Arguing as above, we can suppose that
$0\leq w_j(\sigma)=t_jw_{\rho,r}^{\varepsilon}(\sigma)<\rho_\epsilon$
for $j$ large enough and any $\sigma\in \M$. Thus, by (\ref{conto1})
and (\ref{epsilon}) we get
\begin{multline}
{\frac{\Psi(w_j)}{\Phi(w_j )}  \geq
2\frac{\alpha_0M t_j^2\operatorname{Vol}_g(A_{\varepsilon r}^{\rho}(\sigma_0))+{\displaystyle\int_{A_{r}^{\rho}(\sigma_0)\setminus A_{\varepsilon r}^{\rho}(\sigma_0)} \alpha(\sigma)F(
t_jw_{\rho,r}^{\varepsilon}(\sigma))\,dv_g}}{t_j^2\|w_{\rho,r}^{\varepsilon}\|^2}}\\
{ \geq 2\alpha_0\frac{
M\operatorname{Vol}_g(A_{\varepsilon r}^{\rho}(\sigma_0))+{(\ell-\epsilon)
\displaystyle\int_{A_{r}^{\rho}(\sigma_0)\setminus A_{\varepsilon r}^{\rho}(\sigma_0)} |w_{\rho,r}^{\varepsilon}(\sigma)|^2\,dv_g}}{\|w_{\rho,r}^{\varepsilon}\|^2}}, \label{eps2}
\end{multline}
provided $j$ is sufficiently large.

Choosing $M>0$ large enough, say
\[
M>\max\left\{0, -\frac{2\ell}{\operatorname{Vol}_g(A_{\varepsilon r}^{\rho}(\sigma_0))}\int_{A_{r}^{\rho}(\sigma_0)\setminus A_{\varepsilon r}^{\rho}(\sigma_0)}|w_{\rho,r}^{\varepsilon}(\sigma)|^2\,dv_g\right\},
\]
and $\epsilon>0$ small enough so that
\[
\epsilon \int_{A_{r}^{\rho}(\sigma_0)\setminus A_{\varepsilon r}^{\rho}(\sigma_0)}|w_{\rho,r}^{\varepsilon}(\sigma)|^2\,dv_g<M\frac{\operatorname{Vol}_g(A_{\varepsilon r}^{\rho}(\sigma_0))}{2}+\ell \int_{A_{r}^{\rho}(\sigma_0)\setminus A_{\varepsilon r}^{\rho}(\sigma_0)}|w_{\rho,r}^{\varepsilon}(\sigma)|^2\,dv_g,
\]
by (\ref{eps2}) we get
\begin{align*}
\frac{\Psi(w_j)}{\Phi(w_j)} &\geq 2 \alpha_0 \frac{
M \operatorname{Vol}_g (A_{\varepsilon r}^\rho (\sigma_0)) + (\ell - \epsilon) \displaystyle\int_{A_r^\rho(\sigma_0) \setminus A_{\varepsilon r}^\rho(\sigma_0)} |w_{\rho,r}^\varepsilon (\sigma)|^2 \, dv_g
}{
\|w_{\rho,r}^\varepsilon\|^2
} \\
&\geq \frac{2 \alpha_0}{\|w_{\rho,r}^\varepsilon\|^2} \left( M \operatorname{Vol}_g ( A_{\varepsilon r}^\rho (\sigma_0)) + \ell \int_{A_r^\rho(\sigma_0) \setminus A_{\varepsilon r}^\rho(\sigma_0)} |w_{\rho,r}^\varepsilon (\sigma)|^2 \, dv_g \right. \nonumber \\
&\quad \left. - \varepsilon \int_{A_r^\rho(\sigma_0) \setminus A_{\varepsilon r}^\rho(\sigma_0)} |w_{\rho,r}^\varepsilon (\sigma)|^2 \, dv_g \right) \nonumber \\
&= \alpha_0 \frac{M}{\|w_{\rho,r}^\varepsilon\|^2} \operatorname{Vol}_g(A_{\varepsilon r}^\rho(\sigma_0)), \nonumber
\end{align*}
for $j$ large enough. Also in this case the arbitrariness of $M$
gives assertion~(\ref{wjBlu}).

Now, note that
\[
\|w_j\|=t_j\,\|w_{\rho,r}^{\varepsilon}\|\to 0,
\]
as $j\to +\infty$, so that for $j$ large enough we have
\(
\|w_j\|< \sqrt{2}\bar \gamma.
\)
Thus
\begin{equation}\label{wj1}
w_j\in \Phi^{-1}\big((-\infty, \bar \gamma^2)\big)\,,
\end{equation}
provided $j$ is large enough. Also, by (\ref{wjBlu}) and the fact that
$\lambda>0$
\begin{equation}\label{wj2}
\mathcal
J_{\lambda}(w_j)=\Phi(w_j)-\lambda\Psi(w_j)<0,
\end{equation}
for $j$ sufficiently large.

Since $w_\lambda$ is a global minimum of the restriction of $\mathcal
J_{\lambda}$ to $\Phi^{-1}\big((-\infty,\bar \gamma^2)\big)$ (see (\ref{minimum})), by (\ref{wj1}) and (\ref{wj2})
we conclude that
\begin{equation*}\label{nontrivial}
\mathcal J_{\lambda}(w_\lambda)\leq \mathcal
J_{\lambda}(w_j)<\mathcal J_{\lambda}(0)\,,
\end{equation*}
so that $u_\lambda\not\equiv 0$ ($\mathcal J_{\lambda}(0)=0$) in $\h$.\par
 Thus, $w_\lambda$ is a
non--trivial weak solution of problem~\eqref{problema}. The arbitrariness
of $\lambda$ gives that $w_\lambda\not \equiv 0$ for any
$\lambda\in (0, \lambda_{\star})$.

Now, we claim that $\lim_{\lambda\rightarrow
0^+}{\|w_{\lambda}\|}=0.$

For this, let us fix
$\lambda\in (0, \lambda_{\star}{(\bar{\gamma})})$ for some $\bar \gamma>0$. By $\Phi(w_\lambda)<\bar \gamma^2$ one has that
\[
\Phi(w_\lambda)=\frac 1 2\|w_{\lambda}\|^2<\bar
\gamma^2,
\]
that is
\(
\|w_{\lambda}\|<\sqrt{2}\bar \gamma.
\)
Thus, by using relations~(\ref{crescita2}) and \eqref{constants}, one has
\begin{align}\label{bounded}
\left|\int_{\M} \alpha(\sigma)f(w_{\lambda}(\sigma))w_{\lambda}(\sigma)dv_g\right|  &\leq \beta_f\Bigg(
\int_{\M}\alpha(\sigma)|w_\lambda(\sigma)|dv_g+\int_{\M}\alpha(\sigma)|w_\lambda(\sigma)|^qdv_g \Bigg) \nonumber\\
& \leq \beta_f\Bigg(\|\alpha\|_p\|w_\lambda\|_q+\|\alpha\|_\infty\|w_\lambda\|_q^q\Bigg) \\
& < c_q\beta_f\Bigg(\sqrt{2}\|\alpha\|_p\bar \gamma+2^{q/2}c_q^{q-1}\|\alpha\|_\infty\bar\gamma^q\Bigg).\nonumber
\end{align}

Since $w_\lambda$ is a critical point of $\mathcal
J_{\lambda}$\,, then $\langle \mathcal
J_{\lambda}'(w_\lambda), \varphi \rangle=0$, for any $\varphi
\in \h$ and every $\lambda\in (0,\lambda_{\star}{(\bar{\gamma})})$\,. In particular $\langle
\mathcal J_{\lambda}'(w_{\lambda}), w_{\lambda}\rangle=0$, that is
\begin{equation}\label{rangle}
\langle \Phi'(w_{\lambda}), w_{\lambda}\rangle=\lambda\int_{\M}
\alpha(\sigma)f(w_{\lambda}(\sigma))w_{\lambda}(\sigma)dv_g
\end{equation}
for every $\lambda\in (0,\lambda_{\star}{(\bar{\gamma})})$.

Then, setting
\[
\kappa_{\bar\gamma}:=c_q\beta_f\Bigg(\sqrt{2}\|\alpha\|_p\bar \gamma+2^{q/2}c_q^{q-1}\|\alpha\|_\infty\bar\gamma^q\Bigg),
\]
by (\ref{bounded}) and (\ref{rangle}), it follows that
\[
0\leq \|w_{\lambda}\|^2=\langle\Phi'(w_\lambda),w_\lambda\rangle=\lambda\int_{\M}
\alpha(\sigma)f(w_\lambda(\sigma))w_\lambda(\sigma)\,dv_g< \lambda\,\kappa_{\bar \gamma}
\]
for any
$\lambda\in (0,\lambda_{\star}{(\bar{\gamma})})$\,. We get $
\lim_{\lambda\rightarrow 0^+}{\|w_{\lambda}\|}=0$\,, as
claimed. This completes the proof.\par

\begin{remark}\label{ipotesi}
It easily seen that the statements of Theorem \ref{Main1} are valid if $f$ is a nonlinear term such that
\begin{equation*}
\lim_{t\rightarrow 0^{+}}\frac{f(t)}{t}=+\infty.
\end{equation*}
Moreover, we notice that, along the proof of Theorem \ref{Main1}, assumption
\begin{equation*}
\liminf_{t\rightarrow 0^+}\frac{F(t)}{t^2}>-\infty,
\end{equation*}
ensures the existence of a sequence of functions $\big\{w_j\big\}_{j\in
\NN}$ in $\h$ such that
\begin{equation*}
\limsup_{j\to
+\infty}\frac{\Psi(w_j)}{\Phi(w_j)}=+\infty\,.
\end{equation*}
It is natural to ask if this condition can be removed changing the variational technique. In the Euclidean setting this fact seems to be possible by using the analytic properties of the function $w_0^{\varepsilon_0}$ and the explicit expression of the geodesic volume of the Euclidean balls.
\end{remark}

\begin{remark}
A direct and meaningful consequence of Theorem \ref{Main1} is given in Corollary \ref{Main3}. More precisely, as a simple computation shows, the statement of Corollary \ref{Main3} are valid for every parameter
\begin{equation*}
\lambda\in \left(0,\frac{1}{4(s-1)c_s^2}\left(\frac{s(s-2)^{s-2}}{\|\alpha\|_{\frac{s}{s-1}}^{s-2}\|\alpha\|_{\infty}}\right)^{1/(s-1)}\right).
\end{equation*}
\end{remark}

\begin{ex}\label{esempio}
Let $(\mathcal{M},g)$ be a complete, non--compact, three dimensional Riemannian manifold
such that conditions $(As_{R}^{H, g})$ and \eqref{c6nuovo} hold.
As a model for $f$ we can take the nonlinearity
\begin{equation*}
f(t)=\sqrt{t},\qquad\text{for all $t \geq 0$}.
\end{equation*}
Furthermore, let $\alpha\in L^{\infty}(\M)\cap L^{3/2}(\M)\setminus\{0\}$ and such that \eqref{proprietabeta} holds. Set
\begin{equation*}
\beta=\frac{2\sqrt[4]{3}}{3\sqrt{3}},
\end{equation*}
and
\begin{equation*}
\lambda_{\star}=\frac{1}{4\beta c_3^2}\sqrt{\frac{3}{\|\alpha\|_{3/2}\|\alpha\|_{\infty}}}.
\end{equation*}
Then, fixing $\theta>0$, Theorem \ref{Main1} ensures that, for every $\lambda\in (0,\lambda_{\star})$,
the following problem
\begin{equation}\label{problemanuovooo22}
\left\{
\mkern-8mu
\begin{array}{ll}
-\Delta_gw+(1+d_g(\sigma_0,\sigma)^\theta)w=\lambda \alpha(\sigma)\sqrt{w} & \mbox{ in } \mathcal{M}\\
w\geq 0 & \mbox{ in } \mathcal{M}
\end{array}\right.
\end{equation}
 \noindent admits at least one non--trivial weak solution $w_{\lambda}\in \h$ such that
\[
\lim_{\lambda\rightarrow 0^+}\|w_{\lambda}\|=0.
\]
Moreover, Theorem 3.1 of \cite{FaraciFarkas} yields $w_{\lambda}\in L^{\infty}(\mathcal{M})$ and
\[
\lim_{d_g(\sigma_0,\sigma)\rightarrow +\infty}w_{\lambda}(\sigma)=0.
\]
\end{ex}

\begin{remark}
Since the real function \(h \colon (0,+\infty) \to \mathbb{R}\) defined by
\begin{equation*}
h(t)=\frac{2}{3\sqrt{t}}=\frac{\displaystyle\int_0^{t}\sqrt{\tau}\,d\tau}{t^2}
\end{equation*}
is strictly decreasing, it follows from \cite[Theorem 1.1]{FaraciFarkas} that, for each $r>0$ there exists an open interval $I_r\subseteq (0,+\infty)$ such that, for every $\lambda\in I_r,$ the problem \eqref{problemanuovooo22}
 has a  weak solution $w_\lambda\in \h,$ satisfying
\begin{equation*}
 \int_{\M}  \left(|\nabla_g w(\sigma)|^2+(1+d_g(\sigma_0,\sigma)^\theta|w(\sigma)|^2\right)  dv_g<r.
 \end{equation*}
 However, differently than in Theorem \ref{Main1}, in \cite[Theorem 1.1]{FaraciFarkas} no explicit form of the interval $I_r$ is given.
\end{remark}

\paragraph{Further Perspectives.} As it is well--known one of the most famous
problems in Differential Geometry is the so--called Yamabe problem. This
problem has a PDE formulation which has been extensively studied by many
authors (see, for instance, \cite{EH1, EH2,KRV}). We intend to study a challenging problem
related to Yamabe--type equations on non--compact Riemannian manifolds by using
a suitable group-theoretical approach based on the Palais criticality principle (see, for related topics, the papers \cite{Molica, Molica2, MP1}) under the asymptotically condition $(As_{R}^{H, g})$ on the Ricci curvature. In particular, motivated by the recent results obtained in \cite{M0,M1,MP3,MoRe} we are interested on the existence of multiple solutions for the following critical equation
\begin{equation*}\label{problemacritical}
\left\{
\begin{array}{ll}
-\Delta_gw+V(\sigma)w=\lambda \alpha(\sigma)|u|^{2^*-2}u+g(\sigma,w) & \mbox{ in } \mathcal{M}\\
w\rightarrow 0 & \mbox{ as } d_g(\sigma_0,\sigma)\rightarrow +\infty,
\end{array}\right.
\end{equation*}
where $g$ is a suitable subcritical term. This problem will be studied in a forthcoming paper.\par
\bigskip
\textbf{Acknowledgments.}
The authors are partially supported
by the Italian MIUR project \emph{Variational methods, with applications to problems in mathematical
physics and geometry} and are members of the \emph{Gruppo Nazionale per
l'Analisi Matematica, la Probabilit\`a e le loro Applicazioni} (GNAMPA)
of the \emph{Istituto Nazionale di Alta Matematica} (INdAM).


\begin{thebibliography}{99}

\bibitem{M0} \textsc{V. Ambrosio, J. Mawhin, and G. Molica Bisci}, \emph{$($Super$)$critical nonlocal equations with periodic boundary conditions}, to appear in Selecta Mathematica.

\bibitem{AMR} \textsc{V. Ambrosio, G. Molica Bisci, and D.~Repov\v{s}}, \emph{Nonlinear equations involving the square root of the Laplacian}, to appear in Discrete Contin. Dyn. Syst. Series S, Special Issue on the occasion of the 60th birthday of Vicentiu D. R\u{a}dulescu.

\bibitem{Anello} \textsc{G. Anello}, \emph{A characterization related to the Dirichlet problem for an elliptic equation}, Funkcialaj
Ekvacioj \textbf{59} (2016), 113--122.

\bibitem{Aubin} \textsc{T. Aubin}, \emph{Probl\`{e}mes isop\'{e}rim\'{e}triques et espaces de Sobolev}, C.R. Acad. Sci. Paris S\'{e}r. A-B, \textbf{280} (5) (1975), 279--281.

\bibitem{BK} \textsc{Z.M. Balogh and A. Krist\'aly}, \emph{Lions--type compactness and Rubik actions on the Heisenberg group},
Calc. Var. Partial Differential Equations \textbf{48} (1995), 89--109.

\bibitem{BCW(MatANN)}
\textsc{T. Bartsch, M. Clapp, and T. Weth},
Configuration spaces, transfer, and 2-nodal solutions of a semiclassical nonlinear Schr\"{o}dinger equation,
Math. Ann. \textbf{338} (2007), no. 1, 147--185.


\bibitem{B1} \textsc{ T. Bartsch and Z.-Q. Wang},
\emph{ Existence and multiplicity results for some superlinear elliptic problems in
$\mathbb{R}^N,$} Comm. Partial Differential Equations \textbf{ 20} (1995), 1725--1741.

\bibitem{bw} \textsc{T. Bartsch and M. Willem},\emph{ Infinitely many nonradial solutions of a Euclidean scalar field
equation}, J. Funct. Anal. \textbf{117} (1993), no. 2, 447--460.

\bibitem{bw2} \textsc{T. Bartsch and M. Willem},\emph{ Infinitely many radial solutions of a semilinear elliptic problem in
$\R^N$}, Arch. Rational Mech. Anal. \textbf{124} (1993), 261--276.

\bibitem{B2} \textsc{T. Bartsch, A. Pankov, and Z.-Q. Wang},\emph{ Nonlinear Schr\"odinger equations with steep potential well},
Comm. Contemp. Math. \textbf{4} (2001), 549--569.


\bibitem{B3} \textsc{T. Bartsch, Z. Liu, and T.  Weth},\emph{ Sign--changing solutions
of superlinear Schr\"odinger equations}, Comm. Partial
Differential Equations \textbf{ 29} (2004), 25--42.

\bibitem{BMPR} \textsc{B. Bianchini, L. Mari, P. Pucci, and M. Rigoli},\emph{ On the interplay among weak and strong maximum principles, compact support principles and Keller--Osserman conditions on complete manifolds}, submitted for publication, pages 209.

\bibitem{brezis} \textsc{H. Brezis}, Functional Analysis, Sobolev Spaces and Partial Differential Equations, Universitext, Springer, New York, 2011.

\bibitem{Cia} \textsc{J. Chabrowski and J. do \'O},\emph{ On semilinear elliptic equation involving concave and convex
nonlinearities}, Math. Nachr. (2002), 55--76.

\bibitem{ClappWeth} \textsc{M. Clapp and T. Weth},\emph{ Multiple solutions of nonlinear scalar field equations},
Comm. Partial Differential Equations \textbf{29} (2004), no. 9-10, 1533--1554.

\bibitem{RG} \textsc{M. do Carmo}, Riemannian geometry. Translated from the second Portuguese edition by Francis Flaherty. Mathematics: Theory \& Applications. Birkh\"{a}user Boston, Inc., Boston, MA, 1992. xiv+300 pp.

\bibitem{GW} \textsc{G. Ev\'{e}quoz and T. Weth},\emph{ Entire solutions to nonlinear scalar field equations with indefinite linear part},
Adv. Nonlinear Stud. \textbf{12} (2012), no. 2, 281--314.


\bibitem{FaraciFarkas} \textsc{F. Faraci and C. Farkas},\emph{ A characterization related to Schr\"{o}dinger equations on Riemannian manifolds}, 	 arXiv:1704.02131.

\bibitem{FFK} \textsc{F. Faraci, C. Farkas, and A. Krist\'{a}ly},\emph{ Multipolar Hardy inequalities on Riemannian manifolds}, to appear in ESAIM - Control, Optimisation and Calculus of Variations.

\bibitem{FK} \textsc{C. Farkas, A. Krist\'{a}ly},\emph{ Schr\"{o}dinger--Maxwell systems on non--compact Riemannian
manifolds}, Nonlinear Analysis: Real World Applications \textbf{31} (2016), 473--491.

\bibitem{F1} \textsc{M.F. Furtado, L.A. Maia, and E.A.B.  Silva},\emph{ On a double resonant problem in $\mathbb{R}^N$},
Differential Integral Equations, \textbf{ 15} (2002), 1335--1344.

\bibitem{F2} \textsc{R.L. Frank and E. Lenzmann},\emph{ Uniqueness and nondegeneracy of ground states for $(-\Delta)^sQ+Q-Q^{\alpha+1}=0$ in
$\R$}, Acta Math. \textbf{210} (2013), no. 2, 261--318.

\bibitem{F3} \textsc{F. Gazzola and V. R\u adulescu},\emph{ A nonsmooth critical point theory approach to some nonlinear
elliptic equations in $\mathbb{R}^N$},
Differential Integral Equations \textbf{ 13} (2000), 47--60.

\bibitem{EH1}
 \textsc{E. Hebey}, Nonlinear analysis on manifolds: Sobolev spaces and inequalities. Courant Lecture Notes in Mathematics, 5. New York University, Courant Institute of Mathematical Sciences, New York; American Mathematical Society, Providence, RI, 1999. x+309 pp.

  \bibitem{EH2}
 \textsc{E. Hebey}, Sobolev spaces on Riemannian manifolds. Lecture Notes in Mathematics, 1635. Springer-Verlag, Berlin, 1996. x+116 pp.

\bibitem{kagi1} \textsc{R. Kajikiya},\emph{ Symmetric mountain pass lemma and sublinear elliptic equations}, J. Differential Equations \textbf{260} (2016), no. 3, 2587--2610.

    \bibitem{kagi3} \textsc{R. Kajikiya}, \emph{A critical point theorem related to the symmetric mountain pass lemma and its applications to elliptic equations}, J. Funct. Anal. \textbf{225} (2005), no. 2, 352--370.

 \bibitem{kagi2} \textsc{R. Kajikiya}, \emph{Multiple existence of non--radial solutions with group invariance for sublinear elliptic equations},
J. Differential Equations \textbf{186} (2002), no. 1, 299--343.

 \bibitem{k0} \textsc{A. Krist\'{a}ly},\emph{ Multiple solutions of a sublinear Schr\"{o}dinger equation}, NoDEA Nonlinear Differential Equations Appl. \textbf{14} (2007), no. 3-4, 291--301.

     \bibitem{Kri} \textsc{A. Krist\'{a}ly},
     \emph{Infinitely many radial and non--radial solutions for a class of hemivariational inequalities}, Rocky Mountain J. Math. \textbf{35} (2005), no. 4, 1173--1190.

\bibitem{KRO} \textsc{A. Krist\'{a}ly, G. Moro\c{s}anu, and D. O'Regan}, \emph{A dimension--depending multiplicity result for a perturbed Schr\"{o}dinger equation}, Dynam. Systems Appl. \textbf{22} (2013), no. 2-3, 325--335.

\bibitem{k1} \textsc{A. Krist\'{a}ly and V. R\u{a}dulescu}, \emph{ Sublinear eigenvalue problems on compact Riemannian manifolds with applications in Emden--Fowler equations}, Studia Math. \textbf{191} (2009), no. 3, 237--246.

\bibitem{KRV} \textsc{A. Krist\'{a}ly, V. R\u{a}dulescu, and Cs. Varga}, Variational principles in mathematical physics, geometry, and economics. Qualitative analysis of nonlinear equations and unilateral problems. Encyclopedia of Mathematics and its Applications, \textbf{136}. Cambridge University Press, Cambridge, 2010. xvi+368 pp.


\bibitem{k2} \textsc{A. Krist\'{a}ly and D.~Repov\v{s}}, \emph{ Multiple solutions for a Neumann system involving subquadratic nonlinearities}, Nonlinear Anal. \textbf{74} (2011), 2127--2132.

 \bibitem{k3} \textsc{A. Krist\'{a}ly and D.~Repov\v{s}},\emph{ On the Schr\"{o}dinger--Maxwell system involving sublinear terms}, Nonlinear Anal. Real World Appl. \textbf{13} (2012), 213--223.

\bibitem{k4} \textsc{A. Krist\'{a}ly and I.J. Rudas},\emph{ Elliptic problems on the ball endowed with Funk--type metrics}, Nonlinear Anal. \textbf{119} (2015), 199--208.

\bibitem{k5} \textsc{A. Krist\'{a}ly and Cs. Varga},\emph{ Multiple solutions for a degenerate elliptic equation involving sublinear terms at infinity}, J. Math. Anal. Appl. \textbf{352} (2009), no. 1, 139--148.

\bibitem{M1} \textsc{J. Mawhin and G. Molica Bisci}, \emph{A Brezis--Nirenberg type result for a nonlocal fractional operator}, J. Lond. Math. Soc. (2) \textbf{95} (2017), 73--93.

\bibitem{Molica} \textsc{G. Molica Bisci},\emph{ A group--theoretical approach for nonlinear Schr\"{o}dinger equations}, submitted for publication, pages 24.

 \bibitem{Molica2} \textsc{G. Molica Bisci}, \emph{Variational problems on the sphere}, Recent Trends in Nonlinear Partial Differential Equations II. Stationary problems, Contemp. Math. \textbf{595} (2013), 273--291.

\bibitem{MP1} \textsc{G. Molica Bisci and P. Pucci}, \emph{Critical Dirichlet problems on $\mathcal{H}$ domains of Carnot groups}, submitted for publication, pages 18.


\bibitem{MRadu2} \textsc{G. Molica Bisci and V. R\u{a}dulescu},\emph{ A sharp eigenvalue theorem for fractional
elliptic equations}, Israel Journal of Math. \textbf{219} (2017), no. 1, 331--351.


\bibitem{MRadu1} \textsc{G. Molica Bisci and V. R\u{a}dulescu},\emph{ Multiplicity results for elliptic fractional equations with subcritical term},  NoDEA Nonlinear Differential Equations Appl. \textbf{22} (2015), no. 4, 721--739.

\bibitem{MoRa} \textsc{G. Molica Bisci and V. R\u{a}dulescu}, \emph{Ground state solutions of scalar field fractional Schr\"{o}dinger equations},
Calc. Var. Partial Differential Equations \textbf{54} (2015), 2985--3008.

\bibitem{Radu3} \textsc{G. Molica Bisci and V. R\u{a}dulescu}, \emph{A characterization for elliptic problems on fractal sets}, Proc.
Amer. Math. Soc. \textbf{143} (2015), 2959--2968.

\bibitem{MoRe} \textsc{G. Molica Bisci and D.~Repov\v{s}}, \emph{Yamabe--type equations on Carnot groups}, Potential Anal. \textbf{46} (2017), 369--383.

\bibitem{MP3} \textsc{G. Molica Bisci, D.~Repov\v{s}, and L. Vilasi}, \emph{Existence results for some problems on Riemannian manifolds}, to appear in Communications in Analysis and Geometry.

\bibitem{GMBSe2} \textsc{G. Molica Bisci and R. Servadei}, \emph{An eigenvalue problem for nonlocal equations}, Bruno Pini Math. Anal. Semin. Vol. 7 (2016), 69--84.

\bibitem{GMBSe} \textsc{G. Molica Bisci and R. Servadei}, \emph{A bifurcation result for non--local fractional equations}, Anal. Appl. \textbf{13}, No. 4 (2015), 371--394.

\bibitem{pala} \textsc{R.S. Palais},\emph{ The principle of symmetric criticality}, Commun. Math. Phys. \textbf{69} (1979), 19--30.

    \bibitem{Pucci} \textsc{P. Pucci}, \emph{Critical Schr\"{o}dinger--Hardy systems in the Heisenberg group}, to appear in Discrete Contin. Dyn. Syst. Ser. S, Special Issue on the occasion of the 60th birthday of Vicentiu D. R\u{a}dulescu, pages 26.

\bibitem{PRS}
\textsc{S. Pigola, M. Rigoli, and A. G. Setti}, Vanishing and finiteness results in geometric analysis. A generalization
of the Bochner technique. Volume 266 of Progress in Mathematics. Birkh\"{a}user Verlag,
Basel, 2008.

\bibitem{Per} \textsc{K. Perera}, \emph{Multiplicity results for some elliptic problems with concave nonlinearities}, J. Differential
Equations \textbf{140} (1997), 133--141.

\bibitem{seminalrabinowitz} \textsc{P.H. Rabinowitz},\emph{ On a class of nonlinear Schr\"{o}dinger equations},  {Z. Angew. Math. Phys.} \textbf{ 43} (1992), 270--291.

  \bibitem{Riccerinew} \textsc{B. Ricceri},\emph{ A characterization related to a two--point boundary value problem}, J. Nonlinear Convex
Anal. \textbf{16} (2015), 79--82.

    \bibitem{Ricceri} \textsc{B. Ricceri},\emph{ A further three critical points theorem}, Nonlinear Analysis \textbf{71} (2009), no. 9, 4151--4157.

\bibitem{VPricceri} \textsc{B. Ricceri},\emph{ A general variational principle and some of its applications},
 J. Comput. Appl. Math., \textbf{ 113}  (2000), 401--410.

\bibitem{struwe} \textsc{M. Struwe}, Variational methods,
Applications to nonlinear partial differential equations and Hamiltonian systems,
\emph{ Ergebnisse der Mathematik und ihrer Grenzgebiete},  3,
\emph{ Springer Verlag}, Berlin--Heidelberg (1990).

\bibitem{Wi} \textsc{M. Willem}, {Minimax Theorems}, Birkh\"auser, Basel (1999).

\bibitem{Gigio1} \textsc{L. S. Yu}, \emph{Nonlinear $p$--Laplacian problems on unbounded domains}, Proc. Amer. Math. Soc.
\textbf{115} (1992), 1037--1045.

\end{thebibliography}
\end{document}